\documentclass[12pt,reqno]{amsart}

\usepackage{color}
\usepackage{amsmath, amsthm, amssymb}
\usepackage{amsfonts}
\usepackage[ansinew]{inputenc}
\usepackage[dvips]{epsfig}
\usepackage{graphicx}
\usepackage[english]{babel}
\usepackage{hyperref}
\theoremstyle{plain}
\newtheorem{thm}{Theorem}[section]
\newtheorem{cor}[thm]{Corollary}
\newtheorem{lem}[thm]{Lemma}
\newtheorem{prop}[thm]{Proposition}
\newtheorem{claim}[thm]{Claim}

\theoremstyle{definition}

\theoremstyle{remark}
\newtheorem{rem}[thm]{Remark}

\numberwithin{equation}{section}

\newcommand{\average}{{\mathchoice {\kern1ex\vcenter{\hrule height.4pt
width 6pt depth0pt} \kern-9.7pt} {\kern1ex\vcenter{\hrule
height.4pt width 4.3pt depth0pt} \kern-7pt} {} {} }}
\newcommand{\ave}{\average\int}
\def\R{\mathbb{R}}
\newcommand{\I}{{\rm I}}
\newcommand{\cI}{{\rm I}}

\begin{document}

\title[$C^{\sigma+\alpha}$ regularity for concave equations with rough kernels]{$C^{\sigma+\alpha}$ regularity for concave nonlocal fully nonlinear elliptic equations with rough kernels}
\author{Joaquim Serra}

\address{Universitat Polit\`ecnica de Catalunya, Departament de Matem\`{a}tica  Aplicada I, Diagonal 647, 08028 Barcelona, Spain}

\maketitle

\begin{abstract}
We establish $C^{\sigma+\alpha}$ interior estimates for concave nonlocal fully nonlinear equations of order $\sigma\in(0,2)$ with rough kernels.
Namely, we prove that  if $u\in C^{\alpha}(\R^n)$ solves  in $B_1$ a concave translation invariant equation with kernels in $\mathcal L_0(\sigma)$,
then  $u$ belongs to $C^{\sigma+\alpha}(\overline{ B_{1/2}})$, with an estimate.
More generally, our results  allow the equation to depend on $x$ in a $C^\alpha$ fashion.

Our method of proof combines a Liouville theorem and a blow-up (compactness) procedure.
Due to its flexibility, the same method can be useful in different regularity proofs for nonlocal equations.
\end{abstract}

\section{Introduction and results}

In the paper \cite{CS3}, Caffarelli and Silvestre established the $C^{\sigma+\alpha}$ interior regularity for concave translation invariant nonlocal  fully nonlinear equations of order $\sigma\in(0,2)$ with smooth kernels.
This result extended the classical $C^{2,\alpha}$ interior estimates  for concave second order elliptic equations of Evans \cite{evans} and Krylov \cite{krylov} to the context of integro-differential equations.
The main result in \cite{CS3} states  that if $u\in L^\infty(\R^n)$ satisfies $\inf_a L_a u=0$ in $B_1$ and $L_a\in \mathcal L_2(\sigma)$ for all $a$, then $u\in C^{\sigma+\alpha}\bigl(\overline{B_{1/2}}\bigr)$, with an estimate.

The ellipticity class $\mathcal L_2=\mathcal L_2(\sigma)$ is defined as the set of
all linear translation invariant operators of the form
\begin{equation}\label{operator}
 L_a u = \int_{\R^n} \frac{1}{2}\bigl(u(x+y)+u(x-y)-2u(x)\bigr) K_a(y)\,dy,
 \end{equation}
where $K_a$ are even kernels satisfying
\begin{equation}\label{L0}
0<\frac{\lambda (2-\sigma)}{ |y|^{n+\sigma}} \le K_a(y) \le \frac{\Lambda (2-\sigma)}{ |y|^{n+\sigma}}
\end{equation}
and, in addition, with all its second order partial derivatives satisfying the following scaling invariant bounds away from the origin:
\begin{equation}\label{L2}
[ K_a ]_{C^2(\R^n\setminus B_\rho)} \le  \Lambda (2-\sigma) \rho^{-n-\sigma-2} \quad \mbox{for all }\rho>0.
\end{equation}

The class $\mathcal L_2$ is a subclass of the class $\mathcal L_0$, where $\mathcal L_0$ is formed by all operators of the form \eqref{operator} that satisfy \eqref{L0} but not necessarily \eqref{L2}.
The bounds by above and by below in \eqref{L0} allow the kernels in $\mathcal L_0$ to be very oscillating and irregular, and that is why they are referred to as {\em rough kernels}.

After the paper \cite{CS3}, the following two main questions in the regularity theory of concave nonlocal fully nonlinear elliptic equations remained open.

A first open question was to determine weather the same $C^{\sigma+\alpha}$ estimates held also for non-smooth kernels.
In this direction, to prove the interior $C^{\sigma+\alpha}$ regularity for the equation
\begin{equation}\label{pucci}
 M^-_{\mathcal L_0} u =  0 \quad \mbox{in }B_1.
\end{equation}
is the  fifth open problem listed in the wiki of Nonlocal Equations \cite{wiki}.
Recall that the  extremal operator for the class $\mathcal L_0$ is defined as
\begin{equation}\label{ML0}
 M^-_{\mathcal L_0 } u(x) = \inf_{L\in \mathcal L_0} Lu(x) = \int_{\R^n} \left\{ \lambda\bigl(\delta^2 u(x,y)\bigr)^+ - \Lambda \bigl(\delta^2 u(x,y)\bigr)^-\right\} \, \frac{2-\sigma}{|y|^{n+\sigma}}\,dy.
 \end{equation}
Here, and throughout the article, we use the following notation for second order incremental quotients
\[ \delta^2u(x,y)  =  \frac{1}{2}\bigl(u(x+y)+u(x-y)-2u(x)\bigr).\]

The equation \eqref{pucci} is arguably the canonical example of concave equation of order $\sigma$. As given in \eqref{ML0}, $M^-_{\mathcal L_0}$ has a simple ``closed expression'', similar to
\[M^-(D^2 u) = \lambda {\rm tr }(D^2 u)^+ - \Lambda{\rm tr }(D^2 u)^- \]
for the second order Pucci. Such a closed expression is not available for $M^-_{\mathcal L_2}$.
However, the equation   \eqref{pucci}  is not covered by the theory in \cite{CS3} since it is elliptic with respect to $\mathcal L_0$ but not  with respect to $\mathcal  L_2$.

A second natural  question that remained open after the paper \cite{CS3} was to prove a $C^{\sigma+\alpha}$ Schauder type estimate for non translation invariant equations with $C^\alpha$ dependence on $x$.
In the case of second order fully nonlinear elliptic equations, this Schauder estimate is a classical result. Is is proved by exploiting the fact that in a small neighborhood of a given point the equation is a small perturbation of a translation invariant equation ---this is the nonlinear perturbation method of Caffarelli \cite{Caff}.
To prove $C^{\sigma+\alpha}$ regularity for equations of order $\sigma\in (0,2)$, the same method does not work essentially because if $u$ is a function with a zero of order $\sigma+\alpha$ at $x=0$, the scaling  $\rho^{-\sigma-\alpha} u(\rho\,\cdot\,)$, $\rho\ll 1$ typically results in a growth of the type $|x|^{\sigma+\alpha}$ at infinity, which is not integrable against the tails of the kernel.
This difficulty will be described in more detail later on in the introduction.

In this paper we answer the previous two questions.
More precisely,  we establish existence, uniqueness, and $C^{\sigma+\alpha}$ interior regularity, for nonlocal Dirichlet problems of the form
\begin{equation}\label{dir-pb}
\begin{cases}
\cI(u,x)= 0 \quad & \mbox{in }B_1\\
u=g \quad & \mbox{in }\R^n \setminus B_1,
\end{cases}
\end{equation}
where $\cI$ is a concave operator, elliptic with respect to $\mathcal L_0$, and depending on $x$ in $C^\alpha$ fashion ---see assumptions \eqref{nontrinv}-\eqref{nontrinv1}-\eqref{nontrinv22}-\eqref{nontrinv2} below.
We prove that if $g\in C^\alpha(\R^n\setminus B_1)$ (with $\alpha$ small), then there exists a unique viscosity solution to the problem \eqref{dir-pb}, which is  $C^{\sigma+\alpha}$ in the interior of $B_1$ ---with an estimate in $\overline{B_{1/2}}$.

For equations with rough kernels, our assumption on the {\em complement data} (or {\em exterior data}) $g\in C^\alpha$ can not be weakened to $g\in L^\infty$ ---as in \cite{CS3}.
Indeed,  in the paper we find a sequence of functions $u_m \in C(\R^n)$ that solve in the viscosity sense $M^+_{\mathcal L_0} u_m= 0$ in $B_1$ and satisfy  $\|u_m\|_{L^\infty(\R^n)} =  1$ but   $\|u_m\|_{C^{\sigma+\alpha}(\R^n)} \nearrow +\infty$ as $m\to \infty$ for all $\alpha>0$.
Hence, a $C^{\sigma+\alpha}$ interior estimate  can not hold for any $\alpha>0$ with merely bounded complement data.
To construct such sequence $u_m$ we exploit the strong sensitivity of nonlocal operators with rough kernels to quickly oscillating complement data, to the point that interior regularity
can be broken ``from the exterior'' by choosing  very oscillating exterior data.
However, the more regular the tails of the kernel are, the less sensitive to far oscillations.
In this direction, we prove that when the kernels belong to the class $\mathcal L_\alpha$ ---a scaling invariant class of  $C^\alpha$ kernels--- solutions to concave equations with merely bounded complement data do have $C^{\sigma+\alpha}$ interior regularity.

A main difficulty of nonlocal operators with rough kernels is that, as said above, they are very sensitive to oscillations in the complement data.
This does not happen with smooth kernels because high frequency oscillations balance out when they are integrated against a kernel with smooth tails.
This idea is recurrently exploited in the proofs of \cite{CS3},  essentially by transferring derivatives from the function to the (smooth) kernels with a sort of integration by parts.
Since we can not do the same with rough kernels, we need a different approach.


Similarly as in the concave case,
the  $C^{1+\gamma}$ regularity for general nonlocal fully nonlinear equations was first established for smooth kernels, and only posteriorly extended to rough kernels.
In \cite{CS}, Caffarelli  and Silvestre obtained $C^{1+\gamma}$ interior estimates for these equations in the intermediate class of kernels $\mathcal L_1$ ---those satisfying \eqref{L2} with $C^2$ replaced by $C^1$ and $\rho^{-n-\sigma-2}$ replaced by $\rho^{-n-\sigma-1}$.
It was Kriventsov \cite{K}  to establish the $C^{1+\gamma}$ estimates for elliptic equations of order $\sigma>1+\gamma$ with rough kernels, that is, for $\mathcal L_0$.
The proof in \cite{K} combines  a new estimate for solutions with Lipchitz complement data  and perturbative (compactness) arguments \`a la \cite{CS2}.

Later, in  \cite{Sparab}, we gave a new proof of the result in \cite{K}, extending it also to the parabolic case.
The key idea of this new proof was to deduce the interior regularity from a Liouville theorem, via a blow-up (compactness) argument.
In the present paper, we refine and improve significantly this method of proof from \cite{Sparab} in order to obtain the $C^{\sigma+\alpha}$ estimates for concave equations.
Moreover, the methods of this paper are flexible enough to be applied in other contexts. For instance, the ideas we introduce here ---suitably adapted--- are crucial in the paper \cite{RSboundary}, by Ros-Oton and the author, where the boundary regularity (of order $1+s+\alpha$)
for fully nonlinear elliptic integro-differential equations of order $2s$ is stablished.

As said above, the results of this paper apply to non translation invariant equations with $C^\alpha$ dependence on $x$.
More precisely, while in \cite{CS, CS2, K} the kernels depend only in $y$ ---i.e.  $K_a= K_a(y)$ as in \eqref{operator}---, here we include kernels $K_a(x,y)$ which are $C^\alpha$ in the variable $x$ (in the appropriate sense) and rough in the variable $y$.

We consider concave operators of the form
\begin{equation}\label{nontrinv}
\cI(u,x) :=  \inf_{a\in \mathcal A} \left( \int_{\R^n} \delta^2 u(x,y) K_a(x,y)\,dy  + c_a(x) \right),
\end{equation}
where $\mathcal A$ is some index set.
We assume that  for all $a\in \mathcal A$,    for all $x$ and $x' $ in $\R^n$, and  for all $r>0$ we have
\begin{equation}\label{nontrinv1}
\frac{ \lambda(2-\sigma)}{ |y|^{n+\sigma}}  \le K_a(x,y) \le \frac{ \Lambda(2-\sigma)}{|y|^{n+\sigma}} , \end{equation}
\begin{equation}\label{nontrinv22}
\int_{B_{2r}\setminus B_r}  \bigl| K_a(x, y)-K_a(x', y)\bigr| dy   \le  A_0 |x-x'|^\alpha \frac{2-\sigma}{r^\sigma},
\end{equation}
and
\begin{equation}\label{nontrinv2}
\quad \|c_a\|_{C^{\alpha}(B_1)}\le C_0,
\end{equation}
where $\lambda\le \Lambda$, $A_0$ and $C_0$ are given  constants.

The following is the main result of the paper.
\begin{thm} \label{thm}
Let $\sigma\in (0,2)$, and $\lambda$, $\Lambda$, $A_0$, and $C_0$ be given constants with $0<\lambda\le \Lambda$.
Then, there exists $\bar\alpha>0$ depending only on $n$, $\sigma$, $\lambda$, $\Lambda$ such that the following statement holds.

Let $\alpha\in(0 ,\bar \alpha)$ such that $\sigma+\alpha$ is not an integer. Assume that  $u\in  C^{\sigma+\alpha}(B_1)\cap C^\alpha(\R^n)$ is a solution of
\[ \cI(u,x) = 0 \quad \mbox{in }B_1,\]
where $\cI$ is of the form \eqref{nontrinv} and satisfying \eqref{nontrinv1}, \eqref{nontrinv22}, and \eqref{nontrinv2}.
We then have
\[ \|u\|_{C^{\sigma+\alpha}(B_{1/2})}\le C(C_0 + \|u\|_{C^\alpha(\R^n)} ),\]
where $C_0$ is the constant from \eqref{nontrinv2} and where $C$ depends only on $n$, $\sigma$,  $\alpha$, $\lambda$,  $\Lambda$,  and $A_0$.
\end{thm}

Some comments are in order.
\begin{itemize}
\item Theorem \ref{thm} is stated as an  {\em a priori} estimate: we assume that $u\in C^{\sigma+\alpha} (B_1)$ (with no quantitative control on the norm)  and we obtain a $C^{\sigma+\alpha}$ estimate in $\overline {B_{1/2}}$.
From this a priori estimate, by using the regularization procedure of Section \ref{sec:approx},
we will deduce the existence and uniqueness of a (classical)  $C^{\sigma+\alpha}$ solution to the convex equation $\cI(u,x)=0$ in $B_1$ with given $C^\alpha$ exterior data ---see Theorem \ref{thm2}.

\item As said above, the estimate of Theorem \ref{thm} would be false is we replaced $\|u\|_{C^\alpha(\R^n)}$ in its right hand side
by $\|u\|_{L^\infty(\R^n)}$ ---see Section \ref{SecCounter}.

\item With minor changes in the proofs we can show the dependence of $C$ only on a lower bounds for $\sigma$ and for the gap between $\sigma+\alpha$ and its integer part.
To do it, we can modify the proof  of Proposition \ref{propmain}, adding an additional sequence of orders $\sigma_k\in [\sigma_0,2]$ as in \cite{Sparab}.
Everything  in the paper is prepared so that this can be done (notice in particular that in the statement of the Liouville theorem in Section 3, the exponent $\bar\alpha$ does not depend on $\sigma$).
However, since the proof of Proposition \ref{propmain} is already quite involved as it is, we have chosen not to do this, not to distract the attention from the real
difficulties of the problem.

\end{itemize}

The following corollary provides with a $C^{\sigma+\alpha}$ interior estimate for  solutions $u$ that are merely bounded in $\R^n$ when the kernels are $C^\alpha$ ---recall that for rough kernels this is not possible.
We introduce the class  $\mathcal L_\alpha$ of operators of the form \eqref{operator} with  kernels $K_a$ satisfying \eqref{L0} and \eqref{L2} with $C^2$ replaced by $C^\alpha$ and $\rho^{-n-\sigma-2}$ replaced by $\rho^{-n-\sigma-\alpha}$ ---note that this is consistent  with the definition of $\mathcal L_2$ and $\mathcal L_1$.
In the case of non translation invariant operators we will require the following regularity condition in the variable $y$:
\begin{equation}\label{nontrinv3}
\bigl[ K_a(x,\,\cdot\,) \bigr]_{C^\alpha(\R^n\setminus B_\rho)} \le \Lambda (2-\sigma) \rho^{-n-\sigma-\alpha} \quad \mbox{for all }\rho>0.
\end{equation}

\begin{cor}\label{Lalpha}
Let $\sigma$, $\lambda$, $\Lambda$, $A_0$,  $C_0$, and $\bar \alpha$ as in Theorem \ref{thm}.

Let $\alpha\in(0 ,\bar \alpha)$ such that $\sigma+\alpha$ is not an integer. Assume that  $u\in  C^{\sigma+\alpha}(B_1)\cap L^\infty(\R^n)$ is a solution of
\[ \cI(u,x) = 0 \quad \mbox{in }B_1,\]
with $\cI$ is defined by \eqref{nontrinv} and satisfying \eqref{nontrinv1}, \eqref{nontrinv22}, \eqref{nontrinv2}, and \eqref{nontrinv3}.
Then,
\[ \|u\|_{C^{\sigma+\alpha}(B_{1/2})}\le C\bigl(C_0 + \|u\|_{L^\infty(\R^n)} \bigr),\]
where $C_0$ is the constant from \eqref{nontrinv2} and  $C$ depends only on $n$, $\sigma$,  $\alpha$, $\lambda$,  $\Lambda$,  and $A_0$.
\end{cor}

In order to give an existence and uniqueness result for non translation invariant equations, we need to introduce a regularization procedure based in the one from \cite{CS3}.
We find regularized equations $\cI^\epsilon(u^\epsilon,x)=0$ that have $C^3$ solutions and that converge to $\cI(u,x)=0$ as $\epsilon \searrow 0$ (in the appropriate sense).
A novelty with respect to \cite{CS3} is that for our non translation invariant equations we do not have a comparison principle between viscosity solutions.
Hence, our set up of Perron's method can not rely in the viscosity comparison principle but rather in a property of  ``classical solvability in tiny balls'' for the regularized equations.

Using the regularization procedure and the a priori estimates of  Theorem \ref{thm} and Corollary \ref{Lalpha} we can prove the following existence and uniqueness result.

\begin{thm}\label{thm2}
Let $\sigma$, $\lambda$, $\Lambda$, $A_0$, $C_0$, $\bar \alpha$, $\alpha$ as in Theorem \ref{thm}.

Consider the nonlinear Dirichlet problem \eqref{dir-pb},
where $\cI$, defined by \eqref{nontrinv}, satisfies \eqref{nontrinv1}, \eqref{nontrinv22}, and \eqref{nontrinv2},  and where  $g$ is a bounded function belonging to $C(\R^n)$.
Assume  that either
\begin{itemize}
\item[(a)] $g \in C^{\alpha}(\R^n\setminus B_1)$
\end{itemize}
or
\begin{itemize}
\item[(b)] $\cI$ satisfies \eqref{nontrinv3}.
\end{itemize}
Then, there exists a classical solution $u\in C^{\sigma+\alpha}(B_1)\cap C(\R^n)$ of \eqref{dir-pb}.
As a consequence, the solution $u$ is the unique viscosity solution to \eqref{dir-pb}.

Moreover, this solution $u$ satisfies, in case {\rm (a)}, the estimate
\[ \|u\|_{C^{\sigma+\alpha}(B_{1/2})}\le C\bigl(C_0 + \|g\|_{C^\alpha(\R^n\setminus B_1)} \bigr),\]
and, in case {\rm (b)}, the estimate
\[ \|u\|_{C^{\sigma+\alpha}(B_{1/2})}\le C\bigl(C_0 + \|g\|_{L^\infty(\R^n)} \bigr),\]
where $C_0$ is the constant from \eqref{nontrinv2} and  $C$ depends only on $n$, $\sigma$,  $\alpha$, $\lambda$,  $\Lambda$, and  $A_0$.
\end{thm}

A key idea in our proofs is to deduce the interior regularity results from a Liouville theorem, by using a blow-up (compactness) argument.
As a general type of proof in PDEs,  proving regularity from a Liouville theorem is a well-known strategy that has been used in a large variety of problems.
However, to our knowledge it had not been applied to fully nonlinear elliptic equations until recently by the author in \cite{Sparab} ---a reason explaining this
may be that for second order equations these type of argument gives nothing new with respect to classical perturbative methods.

In the context of nonlocal equations, this method has two main advantages.
First, the Liouville theorem approach allows us to work with solutions in the whole space ---rather than only in $B_1$, say.
This makes possible to deal with rough kernels: we are not troubled the sensitivity to the exterior data because ``there is no exterior data''.
Second, since we blow up the equation, we typically retrieve a translation invariant limiting equation even in when the original equation is not.
This is what allows us to obtain $C^{\sigma+\alpha}$ Schauder estimates  for non translation invariant equations.
Similarly, we could deal with certain ``lower order terms'' which disappear after blow-up.
For instance, our proof immediately applies to the case of truncated kernels.

The outlines of our strategy of proof are the following.
First, we prove a Liouville theorem for global solutions satisfying a certain geometric growth control at infinity.
To do it,  we essentially apply a  regularity proof to a global solution. Using the scaling of the equation and the growth control, we obtain seminorm estimates in every ball $B_R$ that, letting $R\to \infty$, imply that the global solution is a polynomial.
Second, with the new Liouville theorem at hand,  we  use a blow up contradiction argument to deduce a interior regularity estimate for solutions {\em only} in $B_1$.

The faster is the growth allowed  in the Liouville theorem, the better the regularity we will prove with it.
For instance, to prove $C^{1+\gamma}$ regularity for fully nonlinear elliptic equations
of order $\sigma>1+\gamma$ the required Liouville theorem states:  ``if $u$ is a global solution and $|u(x)|\le 1+|x|^{1+\gamma}$ for all $x\in \R^n$, then $u$ is affine''.
This Liouville theorem is quite easy to prove using the H\"older estimates from \cite{CS}.

Similarly, to obtain $C^{\sigma+\alpha}$ estimates we will need a theorem stating: ``if $u$ is a global solution to a concave equation and $|u(x)|\le 1+|x|^{\sigma+\alpha}$ for all $x\in \R^n$ then $u$ is a polynomial of degree two'' (here we are thinking on the most delicate case $\sigma+\alpha>2$).
The problem now is that it is not so clear how to translate this informal statement into a rigorous one.
The most evident difficulty is that for functions growing at infinity like $|x|^{\sigma+\alpha}$ our equation of order $\sigma$ is meaningless, since the operators cannot be computed at such functions. They grow too fast and they are not integrable against the tails of the kernel  decaying like $|y|^{-n-\sigma}$.

An important point in the paper is to find an appropriate statement for this Liouville theorem, which is given in Theorem \ref{liouville} in Section 3.
Since the equation is meaningless due to the fast growth, Theorem \ref{liouville} is not stated for viscosity solutions to some equation but rather for  functions satisfying the three conditions (i)-(iii) in its statement.
Unlike the equation, these three conditions make sense under the growth $|x|^{\sigma+\alpha}$
and they summarize  the relevant information of ``being solution'' to some concave fully nonlinear equation.

As said above, the ideas of the present paper are quite flexible and can be applied to different situations.
In the beginning of this introduction we have referred to the application to boundary regularity by Ros-Oton and the author \cite{RSboundary}.
In more detail, the main result in \cite{RSboundary} states that if $u\in L^\infty(\R^n)$ is a solution of $\I u = 0$ in $B_1^+$ and $u= 0$ in $B_1^-$, with $\I$ elliptic with respect to the class of homogeneous kernels
\[\left\{  \frac{a(y/|y|)}{|y|^{n+2s}}, \ a(y) = a(-y), \ \lambda \le a\le \Lambda, \  \|a\|_{C^{1+\alpha-s}(S^{n-1})} \le \Lambda \right\},\]
then
\[u(x) - (p\cdot x + b) (x_n)_+^s  = o(|x|^{1+s+\alpha}) \quad \mbox{in } B_1^+\mbox{ for }x \sim 0,\]
for some $p\in \R^n$ and $b\in \R$ bounded and for some $\alpha>0$ small.
This result contains in the limit $s\nearrow1$ the classical boundary regularity theory of Krylov for fully nonlinear elliptic equation of second order.
Following the method from \cite{Sparab} and the present paper, the result of \cite{RSboundary} is obtained using blow-up and compactness from a Liouville theorem ---which, in the case of boundary
regularity is for solutions in $\R^n_+$ growing less than  $|x|^{1+s+\alpha}$ as $x\rightarrow \infty$.
Since $1+s+\alpha$ exceeds $2s$ some of the difficulties that we meet in the boundary regularity result are similar to the ones of this paper, and we can solve them by suitably adapting the ideas of this paper to the boundary regularity context.

Another result in which the methods of the present paper have been very useful is the linear regularity theory for the infinitesimal generator of a general symmetric stable L\'evy process, also by Ros-Oton and the author \cite{RSLevy}.
A main result in \cite{RSLevy} is a interior estimate for all equations of the form
\[  (1-s)\int_{\R^n} \bigl(u(x+y)+u(x-y)-2u(x)\bigr) \frac{d\mu(y/|y|)}{|y|^{n+2s}}= f(x) \quad \mbox{in }B_1,\]
where $f\in C^\alpha(B_1)$ and where $\mu$ is a probability measure on $S^{n-1}$. Solutions $u\in C^\alpha(\R^n)$ to the previous equation are shown to belong to $C^{2s+\alpha}(B_{1/2})$ provided that
the measure $\mu$ is  {\em not} supported on some hyperplane (intersected with $S^{n-1}$). Clearly this is also a necessary condition for regularity, because when $\mu$ is supported on some hyperplane the equation will not regularize in the direction orthogonal to the hyperplane.

After finishing a previous version (published online as a preprint) of this paper, Luis Silvestre let us know about the preprint of Tianling Jin and Jingang Xiong \cite{JX}, where they prove Schauder estimates ($C^{\sigma+\alpha}$) for $L^\infty(\R^n)$ solutions to
concave fully nonlinear equations with smooth kernels in $\mathcal L_2$ and  $C^{\alpha}$ dependence on $x$.
Our  Corollary \ref{Lalpha}  applies  in particular to this situation since $\mathcal L_\alpha \subset \mathcal L_2$.
Their results and ours are independent, with different proofs, and both preprints were published online the same day.
As explained in this introduction, Schauder estimates for non translation invariant equations were a main open issue in nonlocal equations and thus believe that it is of interest to have now two different proofs of these estimates in the case of smooth kernels.

The paper is organized as follows. In Section 2 we state and prove the Liouville theorem that serves to obtain $C^{\sigma+\alpha}$ regularity.
In Section 3 we state and prove  Proposition \ref{propmain} (containing the  compactness argument) and use it to prove  Theorems \ref{thm} and Corollary \ref{Lalpha}.
The regularization procedure and the proof of Theorem \ref{thm2} are given in Section \ref{sec:approx}.
Finally, in Section \ref{SecCounter} we give the counterexamples to $C^{\sigma+\alpha}$  interior regularity under the mere assumption of bounded complement data.

Throughout the paper we will use the following conventions:

\begin{itemize}
\item  Given $\beta >0$ which is not an integer we will denote as $C^\beta$ the space $C^{k,\beta'}$ where $k=\lfloor\beta\rfloor$ is the floor of $\beta$ and $\beta'=\beta-k $.

\item  The square brackets $[\,\cdot\,]$ will stand for seminorms. For example, when $\sigma+\alpha\in(2,3)$, $[u]_{C^{\sigma+\alpha}(B_1)}$ will denote the $C^{\sigma+\alpha-2}$ H\"older seminorm of $D^2u$  .

\item The constants $\lambda$ and $\Lambda$ are sometimes referred to as the ``ellipticity constants''.

\end{itemize}

\section{The Liouville theorem}

In this section we state and prove the Liouville theorem that serves to prove $C^{\sigma+\alpha}$ interior regularity.
\begin{thm}\label{liouville}
Let $\sigma_0\in(0,2)$ and $\sigma\in [\sigma_0,2)$. There is $\bar \alpha>0$ depending only on $n$, $\sigma_0$, and ellipticity constants such that the following statement holds.

Let $\alpha$ and $\alpha'$ be constants satisfying $0<\alpha'<\alpha <\bar\alpha$.
Assume that $u\in C^{\sigma+{\alpha'}}_{\rm loc}(\R^n)$ satisfies the following properties.

\begin{enumerate}
\item[(i)] There exists $C_1>0$ such that for all $\beta\in[0,\sigma+\alpha']$ and for all $R\ge1$ we have
\[  [u]_{C^{\beta}(B_R)}\le C_1 R^{\sigma+ \alpha-\beta}.\]

\item[(ii)] For all $h\in \R^n$ we have
\[ M^-_{\mathcal L_0} \bigr( u(\,\cdot\,+h) -u \bigl)\le 0 \le M^+_{\mathcal L_0}\bigr( u(\,\cdot\,+h) -u \bigl) \quad \mbox{in }\R^n.\]

\item[(iii)] For every nonegative $\mu\in L^1(\R^n)$ with compact support and  $\int_{\R^n} \mu(h)\, dh =1$, we have
\[ M^+_{\mathcal L_0}  \left(\ave u(\,\cdot\,+h) \mu(h)\,dh- u\right)\ge 0\quad \mbox{in }\R^n.\]
\end{enumerate}

Then, $u(x)$ is a polynomial of degree $\nu$,  where $\nu$ is the floor (or integer part) of $\sigma+\alpha$.
\end{thm}

In (iii), and in the rest of the paper, the symbol $\ave$ means average (integral with respect to the measure total mass one  $\mu(h)\, dh$). We write $\ave$ even if we could equivalently write $\int$ as a reminder of the assumption $\int_{\R^n} \mu(h)\, dh =1$.

Throughout the paper, $\alpha'$ will be a constant in $(0,\alpha)$.
We will sometimes require, in addition, that $\nu <\sigma+\alpha'$ where $\nu = \lfloor \sigma+\alpha\rfloor$.
An $\alpha'$ satisfying both conditions exists when $\sigma+\alpha$ is not an integer.
In all the paper, one can think of $\alpha'$ as given explicitly by
\begin{equation}\label{alpha'}
\alpha' := \max\left\{ \frac{\alpha}{2}, \frac{\sigma+\alpha+ \nu}{2}-\sigma\right \}.
\end{equation}

The statement of Theorem \ref{liouville} requires some more detailled explanation.  Note that the  $L^\infty$ growth condition in (i) $\|u\|_{L^\infty(B_R)}\le CR^{\sigma+\alpha}$ is too loose for $M^+_{\mathcal L_0} u(x)$ and $M^-_{\mathcal L_0}u(x)$ to be defined, even though  $u\in C^{\sigma+{\alpha'}}$.

However, the control in (i) implies that for every $\beta\in \bigl(0,\min\{1,\sigma+\alpha'\}\bigr)$ we have
\begin{equation}\label{controlbeta}
\|u(\,\cdot\,+h) -u\|_{L^\infty(B_R)}\le C|h|^\beta R^{\alpha+\sigma-\beta}.
\end{equation}

Therefore, taking $\beta\in \bigl(\alpha, \min\{1,\sigma+\alpha'\}\bigr)$ in \eqref{controlbeta} we find that the  $C^{\sigma+{\alpha'}}$ function  $u(\,\cdot\,+h) -u$ belongs to $L^1(\R^n, \omega_\sigma)$ ---here and throughout the paper $\omega_\sigma$ denotes the weight 
\[\omega_\sigma(y) = (1+|y|)^{-n-\sigma}.\]
Thus $M^+_{\mathcal L_0}$ and $M^-_{\mathcal L_0}$ of  $u(\,\cdot\,+h) -u$ are well defined pointwise, and
the inequalities in (ii) are meaningful in the classical sense.

Likewise, the function  $\ave u(\,\cdot\,+h) \mu(h)\,dh- u\,$ is $C^{\sigma+{\alpha'}}$ and belongs to  $L^1(\R^n, \omega_\sigma)$  ---recall that when $\mu$ has compact support. Thus,  the inequality in assumption (iii) is ---also in this case--- meaningful in the classical sense.

\begin{rem}
When $\sigma\le 1$ the proof of this Liouville theorem simplifies significantly and the assumption (iii) is not needed.
In this case, the theorem follows from iterating  the $C^\gamma(\overline{B_{1/2}})$ estimate in \cite[Theorem 12.1]{CS} for solutions $v\in L^\infty(B_1)\cap L^1(\R^n, \omega_\sigma)$ to the two viscosity inequalities $M^-_{\mathcal L_0} v \le 0 \le M^+_{\mathcal L_0} v $ in $B_1$).
Applying this $C^\gamma$ estimate to incremental quotients of $u$ at every scale and iterating  (like in the proof of $C^{1+ \gamma}$ regularity for fully nonlinear equations) we obtain
\[  [u]_{C^{1+  \gamma}( B_R)} \le CR^{\alpha+\sigma-1-\gamma}.\]
Then, since $\sigma \le 1$ the conclusion of the theorem holds taking $R\nearrow \infty$ provided that $\alpha<\gamma$.
For a very similar argument in the parabolic setting see \cite{Sparab}.
\end{rem}

\begin{proof}[Proof of Theorem \ref{liouville}]
The result for all $C_1>0$ trivially follows from the result for $C_1=1$. Thus, in all the proof we assume that $C_1=1$.

In this proof we follow to a large extend the exposition in the lecture notes of Silvestre \cite{SLectureNotes}, where an insightful sketchy version of the $C^{\sigma+\alpha}$ regularity proof from \cite{CS3} is given.
In the present Liouville theorem setting, however,  the same ``simplified'' argument (with few modifications) provides with a short complete proof. This is because since the equation holds in all the space there is no need to truncate functions, avoiding many technical complications.

We want to prove that for some $\bar \alpha$
depending only on $n$, $\sigma_0$, $\lambda$, and $\Lambda$  (but {\em not} on ${\alpha'}$ nor $\alpha$) we have
\begin{equation}\label{last}
[u]_{C^{\sigma+\bar \alpha}(B_R)} \le CR^{\alpha-\bar\alpha},
\end{equation}
with $C$ independent of $R$.
Once this will be proved, since  $\alpha<\bar\alpha$,  sending $R$ to infinity the theorem will follow.

Let us define
\[P(x):=  \int_{\R^n}\bigl(\delta^2u(x,y)-\delta^2u(0,y)\bigr)^+\, \frac{2-\sigma}{|y|^{n+\sigma}}\,dy \]
and
\[N(x):=  \int_{\R^n}\bigl(\delta^2u(x,y)-\delta^2u(0,y)\bigr)^- \,\frac{2-\sigma}{ |y|^{n+\sigma}}\,dy .\]
Using (i) ---recall that $C_1=1$--- we find that $P$ and $N$  are $C^{\alpha'}$ and satisfy
\begin{equation}\label{growthPN0}
  0\le  P \le CR^{\alpha}\quad \mbox{and}\quad   0\le  N \le CR^{\alpha} \quad \mbox{in }B_R,
\end{equation}
for all $R\ge 1$, with $C$ universal (meaning that it depends only on $n$, $\sigma_0$, $\lambda$, and $\Lambda$).
Indeed, let us prove \eqref{growthPN0} when  $\nu =\lfloor \sigma+\alpha\rfloor =2$ (the cases $\nu=0$ and $\nu=1$ are very similar).
Using that $[u]_{C^{\sigma+\alpha'}(B_2)}\le 1$ and that
$[u]_{C^\beta(B_R)}\le R^{\sigma+\alpha-\beta}$ we obtain, taking $\beta\in \bigl(\alpha, \min\{1,\sigma+\alpha'\}\bigr)$, that
\[ \bigl|\delta^2u (x, y)- \delta^2u (x', y) \bigr| \le
\begin{cases}
 C|y|^2 d^{\sigma+\alpha'-2}  \ &\quad \mbox{for } y\in B_d\\
 Cd^2 |y|^{\sigma+\alpha'-2} &\quad \mbox{for } y \in B_{1/2}\setminus B_d\\
 Cd^\beta R^{\sigma+\alpha-\beta} &\quad \mbox{for } y \in B_R\setminus B_{1/2}.
 \end{cases}
 \]
Therefore,
\begin{equation}\label{1234}
\begin{split}
|P(x)-P(x')| &\le  \int_{\R^n} \bigl|\delta^2u(x,y)-\delta^2u(x',y)\bigr|\, \frac{2-\sigma}{|y|^{n+\sigma}}\,dy
\\
&\le  C\int_{\R^n}\bigl( |y|^2 d^{\sigma+\alpha'-2} \chi_{B_d}(y) + d^2 |y|^{\sigma+\alpha'-2}  \chi_{B_{1/2}\setminus B_d }(y)\ +
\\
&\hspace{38mm} + d^\beta |y|^{\sigma+\alpha-\beta}\chi_{\R^n\setminus B_{1/2}}(y) \bigl)\, \frac{2-\sigma}{|y|^{n+\sigma}}\,dy
\\
&\le C (d^{\alpha'}+d^\beta)\le Cd^{\alpha'}.
\end{split}
\end{equation}

This shows that $P\in C^{\alpha'}(B_1)$. Taking $x'=0$ in \eqref{1234} we obtain the bound by above for $P$ in $B_1$ of \eqref{growthPN0}. To prove the same bound  in $B_R$ for all $R\ge 1$ we  use rescaling.
 Given $\rho>0$ we consider the rescaled function
 \[\bar u(x) = \rho^{-\sigma-\alpha} u(\rho x)\]
It is immediate to verify that $\bar u$ satisfies the same assumptions (i), (ii), and (iii) as $u$. In particular the constant $C_1$ in (i) for $\bar u$ is the same as that of
$u$, that is $C_1=1$. Then, as we have proved before for $u$, we have
\[0 \le  \int_{\R^n}\bigl(\delta^2\bar u(x,y)-\delta^2\bar u(0,y)\bigr)^+\, \frac{2-\sigma}{|y|^{n+\sigma}}\,dy \le C \quad \mbox{for all }x\in B_1.\]
Translating this from $\bar u$ to $u$ we obtain that  $0\le P\le C\rho^\alpha$ in $B_\rho$ and hence letting $\rho=R$ we obtain the bound for $P$ in $B_R$ of \eqref{growthPN0}.
The bounds for $N$ in \eqref{growthPN} are obtained likewise.

Next, dividing $u$ by the universal constant $C$ in \eqref{growthPN0} we may assume
\begin{equation} \label{growthPN}
  0\le  P \le 2^{k\alpha} \le 2^{k\bar \alpha} \quad \mbox{in } B_{2^k}(0) \quad \mbox{for all }k\ge0.
\end{equation}
In order to show that $u \in C^{\sigma+\bar \alpha}$  we will prove that
\begin{equation} \label{dyadicimprovement}
  0\le  P \le 2^{k\bar \alpha} \quad \mbox{in } B_{2^k}(0) \quad \mbox{for all }k\le-1.
\end{equation}
This estimate on $P$ is proved though an iterative improvement on the maximum of $P$ on dyadic balls.

Indeed, our goal is to improve the  bound from above $P\le1$ in  $B_1$ to  $P\le 1-\theta$ in $B_{1/2}$, for some $\theta>0$.
After doing this, we will immediately have \eqref{dyadicimprovement} for all $k\ge1$ for some $\bar\alpha$ small (related  to $\theta$) just by scaling and iterating.
Let us thus concentrate in proving  $P\le 1-\theta$ in $B_{1/2}$.

Let $x_0 \in B_{1/2}$ be such that $P(x_0) = \max_{B_{1/2}} P$. Define the set
\[ A = \{ y : (u(x_0+y)+u(x_0-y)-2u(x_0) - u(y) - u(-y) + 2u(0)) > 0 \}.\]
 In particular we have
\begin{align*}
P(x_0) = \int_{A} \bigl(\delta^2u(x_0,y)-\delta^2u(0,y)\bigr) \frac{2-\sigma}{|y|^{n+\sigma}} dy, \\
N(x_0) = \int_{\R^n \setminus A} \bigl(\delta^2u(x_0,y)-\delta^2u(0,y)\bigr) \frac{2-\sigma}{|y|^{n+\sigma}} dy.
\end{align*}

We will take  $\bar \alpha$ very small (depending on $\delta_0$ below) so that \eqref{growthPN} implies
\begin{equation} \label{eeeerror}
\int_{\R^n} \bigl(P(y)-1\bigr)^+ \frac{2-\sigma}{|y|^{n+\sigma}} \,dy   \le \delta_0.
\end{equation}

We define the function $v$ as
\[ v(x) := \int_{A} \bigl(\delta^2u(x,y)-\delta^2u(0,y)\bigr) \frac{2-\sigma}{|y|^{n+\sigma}} dy. \]
Note that in particular  $P(x_0)=v(x_0)$.
Let
\begin{equation}\label{bartheta}
\bar \theta = \frac{\lambda}{4\Lambda}
\end{equation}
and define the set
\[\boldsymbol D := \{ x\in B_1\, :\, v \ge (1-\bar \theta)\}.\]

Let us show that, for  $\eta>0$ small enough we have
\begin{equation}\label{Asmall}
|\boldsymbol D|\le (1-\eta) |B_1|.
\end{equation}

Assume by contradiction that $|\boldsymbol D|\ge (1-\eta) |B_1|$ for $\eta$ small to be chosen later.  That is,  $v$ is larger than $(1-\bar\theta)$ in most of  $B_1$.
In that case we consider the function $w$ defined as $v$ but replacing $A$ by $\R^n\setminus A$.
\[ w(x) := \int_{\R^n \setminus A} \bigl(\delta^2u(x,y)-\delta^2u(0,y)\bigr)  \frac{2-\sigma}{|y|^{n+\sigma}}dy. \]
Using (iii), approximating $\chi_{\R^n\setminus A}(y) (2-\sigma) |y|^{-n-\sigma}$ by $L^1$ functions $\mu$ with compact support and using the stability under uniform convergence result for subsolutions  \cite[Lemma 4.3]{CS2} we show that
\[ M^+_{\mathcal L_0} w \geq 0  \quad \mbox{in }\R^n.\]

We observe that  by definition $P-N = v+w$ and that, we have \[0 \leq P - v \leq 1-(1-\bar\theta)\le \bar \theta\quad \mbox{in }\boldsymbol D\] ---here we have used that $P\le 1$ in $B_1$ by \eqref{growthPN} . Note in addition that the assumption (ii) yields
\begin{equation}\label{PNcomparable}
\frac{\lambda}{\Lambda} P(x) \le N(x) \le \frac{\Lambda}{\lambda} P(x).
\end{equation}
Therefore,
\[\begin{split}
 w &= (P-v) -N \le \bar \theta -N \le \bar\theta -\frac{\lambda}{\Lambda}P  \\
&\le \bar \theta - \frac{\lambda}{\Lambda} (1-\bar \theta)\\
&\le - \lambda/\Lambda + 2\bar \theta \le -c  \quad \mbox{ in } \boldsymbol D,
\end{split}\]
where $c= \lambda/2\Lambda >0$. Here we have used \eqref{bartheta}.

We now use the ``half'' Harnack of Theorem 5.1 in \cite{CS3} applied to the function $\bar w = \bigl(w(r\,\cdot\,)+c\bigr)^+$ (with $r>0$ small) to conclude that $w(0)+c \le c/2$.
Indeed, the function $\bar w$ is a subsolution and, by \eqref{growthPN},  it satisfies $0\le  P \le 2^{k\bar \alpha} \quad \mbox{in } B_{2^k/r}(0)$ and $\bar w =0$ in $\boldsymbol D/r$, which covers most of $B_{1/r}$.
Hence, taking both $r$ and $\eta$ small enough we can make $\int_{\R^n} \bar w(y) \omega_\sigma(y) \,dy$ as small as we wish. Thus, using Theorem 5.1  in \cite{CS3} we find that $w(0)+c = \bar w(0)\le c/2$ as promised.
As a consequence we obtain that $w(0)\le -c/2<0$; a contradiction since $w(0)=0$ by definition.
This proves that \eqref{Asmall} holds for some $\eta>0$.

Note now that \eqref{Asmall} is equivalent to
\[ \bigl| \{ x\in B_1\, :\, v \le (1-\bar \theta)\}\bigr| \ge \eta |B_1|.\]
Next, by (iii), approximating $\chi_{A}(y) |y|^{-n-\sigma}$ by $L^1$ functions $\mu$ with compact support and using the stability under uniform convergence result for subsolutions  \cite[Lemma 4.3]{CS2} we show that
\[ M^+_{\mathcal L_0} v \geq 0  \quad \mbox{in }\R^n. \]
Taking now $\delta_0$ small enough in \eqref{eeeerror} and using  the $L^\varepsilon$ Lemma of   Theorem 10.4 in \cite{CS} applied to the function $(1-v)^+$, which nonnegative in all of $\R^n$ and which is an approximate supersolution in $B_{3/4}$, we obtain that
\[ 1-v \ge \bar \theta/C  \quad \mbox{ in all }B_{1/2}.\]
This is equivalent to saying
\[ v \le 1- \bar \theta/C =: 1-\theta  \quad \mbox{ in all }B_{1/2},\]
as we wanted to show.
This proves \eqref{dyadicimprovement}.

We next note that \eqref{dyadicimprovement} implies
\begin{equation}
\label{equationP}
0\le P(x)\le C|x|^{\bar\alpha}  \quad \mbox{for all }x\in B_1.
\end{equation}
Given that \eqref{PNcomparable} holds ---recall that this follows from assumption (ii)--- we similarly obtain that $ 0\le N(x)\le C|x|^{\bar\alpha}$.

Finally we notice that the point $0$ in the definition of $P$ and $N$ can be replaced by any point $z$ in $B_{1/2}$. Therefore,  using that
\[P(h)-N(h) = c(-\Delta)^{\sigma/2} \bigl( u(\cdot +h) -u \bigr)  (0),\] for some constant $c<0$, we have shown ---replacing $0$ be any $z\in B_{1/2}$--- that
\[ \left| (-\Delta)^{\sigma/2}  \bigl( u(\cdot +h) -u \bigr) \right| \le C|h|^{\bar \alpha} \quad \mbox{ in } B_{1/2},\]
for all $h\in B_{1/4}$.
This and the classical $C^{\bar \alpha}$ to $C^{\sigma+\bar\alpha}$ estimate for the Riesz potential $(-\Delta)^{-\sigma/2}$ easily imply that $[u]_{C^{\sigma+\bar \alpha}(B_{1/4})} \le C$.

The same argument repeated at every scale ---replacing $u$ by the rescaled function $\bar u = \rho^{-\sigma-\alpha} u(\rho\,\cdot\,)$ for all $\rho\ge 1$--- yields $[\bar u]_{C^{\sigma+\bar \alpha}(B_{1/4})} \le C$ which after rescaling  gives \eqref{last}.
Then, as explained previously in this proof, the Theorem follows straightforward letting $R\to \infty$.
\end{proof}

\section{Preliminary results and proof of Theorem \ref{thm}}

The following proposition is the core of Theorem \ref{thm}. It is in its proof (by contradiction) where we use the blow-up argument and the Liouville theorem described in the introduction.

\begin{prop} \label{propmain}
Let $\sigma\in (0,2)$. There is $\bar \alpha>0$ (depending only on $\sigma$, ellipticity constants, and dimension) such that the following statement holds.
Given $\alpha\in (0,\bar\alpha)$ let $\nu$ be the floor of $\sigma+\alpha$ and assume that ${\alpha'}\in(0,\alpha)$ satisfies $\nu<\sigma+{\alpha'}<\sigma+\alpha$.
Let $u\in C^{\sigma+{\alpha'}}(\R^n)$ be a solution of
\[
 \inf_{a\in \mathcal A} \bigl( L_a u  + c_a(x) \bigr) = 0 \quad\mbox{ in }B_{1},
\]
where $\{L_a\}\subset \mathcal L_0(\sigma, \lambda,\Lambda)$.
Assume that
\begin{equation}\label{nontrinvP}
\sup_{x\in B_1} |\inf_{a\in \mathcal A} c_a(x)| <+\infty  \quad \mbox{and }\quad  \sup_{a\in \mathcal A} \,[c_a]_{C^{\alpha}(B_1)}\le C_0.
\end{equation}

Then, $u \in C^{\sigma+\alpha}(\overline{B_{1/2}})$ and
\[ [u]_{C^{\sigma+\alpha}(B_{1/2})}\le C\bigl( \|u\|_{C^{\sigma+{\alpha'}}(\R^n)} +  C_0 \bigl),\]
where $C_0$ is the constant from \eqref{nontrinvP} and  $C$ depends only on $n$, $\sigma$, $\alpha$, ${\alpha'}$, $\lambda$, and $\Lambda$.
\end{prop}

We will use the following trivial Claim.
\begin{claim}\label{lemaux}
Let $\beta>0$ and $\beta'\in(0,\beta)$.  Let $\nu = \lfloor \beta \rfloor$ be the floor (or integer part) of $\beta$ and assume that $\nu <\beta'<\beta$.
Let $u$ be a continuous function belonging to $C^{\beta'}(\R^n)$.

If there exists $C_0>0$ such that
\begin{equation}\label{allrz}
\sup_{r>0} \sup_{z\in B_{1/2}} r^{\beta'-\beta}\bigr[ u \bigr]_{C^{\beta'}(B_r(z))} \le C_0 ,
\end{equation}
then
\begin{equation}\label{holderbound}
[u]_{C^{\beta}\left(B_{1/2}\right)} \le  C_0.
\end{equation}
\end{claim}
\begin{proof}
It is enough to prove it for $\nu=0$, that is, $0<\beta'<\beta<1$ since the result for $\nu\ge 1$ follows from this case applied to partial derivatives of $u$.

To prove it, note that \eqref{allrz} implies that for all $z\in B_{1/2}$ and for all $r>0$ we have
\[
\|u(z+\,\cdot\,)-u(z)\|_{L^\infty(B_r)}\le  r^{\beta'}\bigr[ u \bigr]_{C^{\beta'}(B_r(z))} \le r^{\beta'} C_0 r^{\beta-\beta'}  = C_0 r^{\beta}.
\]
Hence \eqref{holderbound} follows.
\end{proof}

We now give the
\begin{proof}[Proof of Proposition \ref{propmain}]
The proof is by contradiction.
If the statement of the proposition is false then, for each integer $k\ge 0$, there exist  $u_k$ and $C_{0,k}$ such that
\begin{itemize}
\item $ \inf_{a\in \mathcal A_k} \left( L_a u  + c_a(x) \right) =  0 $   in $B_{1}$;
\vspace{3pt}
\item $|\inf_{a\in \mathcal A_k} c_a(x)| <+\infty$  \hspace{5pt} and  \hspace{5pt} $\sup_{a\in \mathcal A_k} \,[c_a]_{C^{\alpha}(B_1)}\le C_{0,k}$;
\vspace{3pt}
\item $\| u_k \|_{C^{\sigma+{\alpha'}}(\R^n )} + C_{0,k} \le 1$ (we may  always assume this dividing $u_k$ by  the previous quantity);
\end{itemize}
and
\[ [u_k]_{C^{\sigma+\alpha}(B_{1/2})} \ge k.\]

Using Claim \ref{lemaux} with $\beta= \sigma+\alpha$ and $\beta'= \sigma+{\alpha'}$ we obtain that
\begin{equation}\label{2k2}
\sup_k \sup_ {z\in B_{1/2}} \sup_{r>0} \ r^{{\alpha'}-\alpha}\left[ u_k \right]_{C^{\sigma+{\alpha'}}(B_{r}(z))} = +\infty.
\end{equation}

Next we define
\[
 \theta(r) := \sup_k  \sup_ {z\in  B_{1/2}}  \sup_{r'>r}  (r')^{{\alpha'}-\alpha}\,\bigl[u_k\bigr]_{C^{\sigma+{\alpha'}} \left(B_{r'}(z)\right)} \, ,
\]
The function $\theta$ is monotone nonincreasing and we have $\theta(r)<+\infty$ for $r>0$ since we are assuming that  $[u_k]_{C^{\sigma+{\alpha'}}(\R^n)}\le 1$.
In addition, by \eqref{2k2} we have  $\theta(r)\nearrow +\infty$ as $r\searrow0$.
For every positive integer $m$, by definition of $\theta(1/m)$ there are $r'_m\ge 1/m$, $k_m$, and $z_{m} \in  B_{1/2}$, for which
\begin{equation}\label{nondeg2}
(r'_m)^{{\alpha'}-\alpha}  \bigl[u_{k_m}\bigr]_{C^{\sigma+{\alpha'}} \left(B_{r'_m}(z_m)\right)} \ge \frac{1}{2} \theta(1/m) \ge \frac{1}{2}  \theta(r'_m).
\end{equation}
Here we have used that $\theta$ is non-increasing.
Note we will have $r'_m\searrow0$.

Let $p_{k,z,r}(\cdot\,-z)$ be the polynomial  of degree less or equal than $\nu$  in the variables $(x-z)$ which best fits $u_k$ in $B_r(z)$ by least squares. That is,
\[p_{k,z,r} := {\rm arg\,min}_{p\in \mathcal P_\nu} \int_{B_r(z)} \bigl(u_k(x)-p(x-z)\bigr)^2 \,dx ,\]
where $\mathcal P_\nu$ denotes the linear space of polynomials of degree at most $\nu$ with real coefficients.
From now on in this proof we denote \[p_m = p_{k_m, z_m,r'_m}.\]

We consider the blow up sequence
\begin{equation}\label{eqvm}
 v_m(x) = \frac{u_{k_m}(z_{m} +r'_m x)-p_{m}(r'_m x)}{(r'_m)^{\sigma+\alpha}\theta(r'_m)}.
 \end{equation}
Note that, for all $m\ge 1$ we have
\begin{equation}\label{2}
\int_{B_1(0)} v_m(x) q(x) \,dx =0\quad \mbox{for all } q\in \mathcal P_\nu.
\end{equation}
This is the optimality condition for least squares.

Note also that \eqref{nondeg2} implies the following inequality for all $m\geq1$:
\begin{equation}\label{nondeg35}
\begin{split}
[v_m]_{C^{\sigma+{\alpha'}}(B_1)} &=  (r'_m)^{\sigma+\alpha'} \left[ \frac{u_{k_m}(z_{m} +r'_m \,\cdot\,)-p_{m}(r'_m \,\cdot\,)}{(r'_m)^{\sigma+\alpha}\theta(r'_m)} \right]_{C^{\sigma+\alpha'}\left(B_{r'_m(z_m)}\right)}
\\
&=  \frac{(r'_m)^{{\alpha'}-\alpha}}{\theta(r'_m)} \bigl[ u_{k_m}(z_{m} +r'_m x) \bigr]_{C^{\sigma+\alpha'}\left(B_{r'_m(z_m)}\right)}  \ge 1/2,
\end{split}
\end{equation}
Here we have used that $\nu:= \lfloor \sigma+\alpha\rfloor <\sigma+\alpha'$, and thus
 \[\bigl[ p_m(z_{m} +r'_m x) \bigr]_{C^{\sigma+\alpha'}\left(B_{r'_m(z_m)}\right)}=0,\]
 since $p_m$ is a polynomial of degree $\nu$. Note that it is here were we crucially use the assumption that $\sigma+\alpha$ is not an integer.

Next we want to estimate
\[\begin{split}
[v_{m}]_{C^{\sigma+{\alpha'}}(B_R)} &= \frac{1}{\theta({r'_m})(r'_m)^{\alpha-{\alpha'}} } \bigl[u_{k_m} \bigr]_{C^{\sigma+{\alpha'}}\left(B_{Rr'_m}(z_m)\right)}
\\
&= \frac{R^{\alpha-{\alpha'}}}{\theta({r'_m}) (Rr'_m )^{\alpha-{\alpha'}} }\bigl[u_{k_m}\bigr]_{C^{\sigma+{\alpha'}}\left(B_{Rr'_m}(z_m)\right)} .
\end{split}
\]
To do it, we use the definition of $\theta$ and its monotonicity to obtain the following growth control for the $C^{\sigma+{\alpha'}}$ seminorm of $v_m$
\begin{equation}\label{growthc0}
[v_{m}]_{C^{\sigma+{\alpha'}} (B_R)} \leq CR^{\alpha-{\alpha'}}\quad \textrm{for all}\ \,R\ge 1.
\end{equation}

When $R=1$,   \eqref{growthc0}  implies that $\|v_m- q\|_{L^\infty(B_1)}\le C$, for some $q\in \mathcal P_\nu$. Then, \eqref{2} implies that
\begin{equation}\label{boundedinB1}
\|v_m\|_{L^\infty(B_1)}\le C.
\end{equation}
Then, using \eqref{growthc0} we obtain
\begin{equation}\label{growthc1}
 [v_{m}]_{C^{\beta} (B_R)} \leq CR^{\sigma+\alpha-\beta}.
\end{equation}
for all $\beta\in[0, \sigma+{\alpha'}]$.
Indeed, \eqref{boundedinB1} implies that for every multiindex $l$ with $|l|\le\nu$ there is some point  $x_* \in B_1$  such that  $|D^l v_m(x_*)|\le C$. The existence of such $x_*$ can be shown taking some nonnegative $\eta\in C^\infty_c(B_1)$ with unit mass and observing that
the inequality
\[ \left|\int \eta (x) D^l v_m (x) \,dx\right| \le C \int |D^l \eta| v_m(x)\,dx \le C\]
rules out the two possibilities $D^l v_m >C$ and $D^l v_m <-C$ in all of $B_1$.

Hence, using \eqref{growthc0}, we obtain that for all $l$ with $|l|=\nu$ and $x\in B_R$ we have
\[  |D^l v_m (x) | \le |D^l v_m (x^*)| + C R^{\alpha-{\alpha'}} |x-x^*|^{\sigma+{\alpha'}-\nu}  \le CR^{\sigma+\alpha-\nu}.\]
Iterating the same argument, we then show the corresponding estimate for all $l$ with $0\le |l| \le \nu \nu$.
Then \eqref{growthc1} for all $\beta\in [0,\sigma+{\alpha'}]$ follows by interpolation.

We now claim that, by further rescaling $v_m$ if necessary, we may assume that in addition to \eqref{nondeg35} the following holds
\begin{equation}\label{nondeg35bis}
\sup_{|l|=\nu}{\rm osc}_{B_1} D^l v_m  \ge 1/4,
\end{equation}
where $l$ donates a multiindex.
Indeed, if \eqref{nondeg35} holds then there are $x_m \in B_1$ and $h_m\in B_{1-|x_m|}$ such that
\[ \sup_{|l|=\nu} \frac{\bigl| D^l v_m(x_m+h_m)- D^l v_m(x_m)\bigr|}{ |h_m|^{\sigma+\alpha'-\nu}} \ge 1/4\]
and thus we can consider, instead of $v_m$, the function
\[ \tilde v_m = \frac{v_m(x_m+|h_m|x) - \tilde p_m(x) }{|h_m|^{\sigma+\alpha'}} ,\]
where $\tilde p_m \in \mathcal P_\nu$ is chosen so that $\tilde v_m$ satisfies \eqref{2} (with $v_m$ replaced by $\tilde v_m$).

Note that $\tilde p_m$  is the polynomial that approximates better (in the $L^2$ sense) $v_m(x_m +\, \cdot\,)$ in $B_{|h_m|}(x_m)$ and since $v_m \in C^{\sigma+\alpha'}$ with the control \eqref{growthc0} we have
\[  \bigl | v_m(x_m+|h_m|x) - \tilde p_m(x)\bigr| \le C|h_m|^{\sigma+\alpha'} |x|^{\sigma+\alpha'} .\]
Therefore, $\tilde v_m$ also satisfies \eqref{growthc0} and \eqref{growthc1} (with $v_m$ replaced by $\tilde v_m$).
Note that $\tilde v_m$ would also be of the form \eqref{eqvm} for new $z_m$ and $r'_m$  defined as $z_m+x_m$ and $|h_m|r'm$, respectively ---where we use that $\theta(|h_m|r'_m)\ge \theta(r'_m)$.

In summary, the new sequence $\tilde v_m$ satisfies the same properties as $v_m$ and, in addition, \eqref{nondeg35bis}, as desired.

Next we prove the following

\vspace{5pt}
\noindent {\em Claim. A subsequence of $v_m$ converges in $C^{(\nu+\sigma+{\alpha'})/2}_{\rm loc}(\R^n)$ to a function $v\in C^{\sigma+{\alpha'}}_{\rm loc}(\R^n)$. This function $v$ satisfies the assumptions of the Liouville-type Theorem \ref{liouville}.}
\vspace{5pt}

The $C^{(\nu+\sigma+{\alpha'})/2}$ uniform convergence on compact sets of $\R^n$ of a subsequence of  $v_m$ to some
$v\in C^{\sigma+{\alpha'}}(\R^n)$ follows from \eqref{growthc1} and the Arzel\`a-Ascoli theorem (and the typical
diagonal sequence trick) ---note that since $\nu <\sigma+\alpha'$ the exponent $(\nu+\sigma+{\alpha'})/2$ is less than
$\sigma+\alpha'$, as required to have compactness in the norm $C^{(\nu+\sigma+{\alpha'})/2}$ of a equibounded sequence
in the stronger norm $C^{\sigma+{\alpha'}}$. The only important fact about the election of the exponent
$(\nu+\sigma+{\alpha'})/2$ is that it is greater that $\nu$ and $\sigma$.

First, passing to the limit \eqref{growthc1} we find that the assumption (i) of Theorem \ref{liouville} is satisfied by this limit function $v$.

Now, each $u_k$ satisfies a concave equation of the type  \eqref{nontrinv}-\eqref{nontrinv1}-\eqref{nontrinv2}.
Thus,   for  every $L^1$ density $d\mu(h)$ with compact support and $\mu(\R^n)=1$  and for $m$ large enough we have
\[
\begin{split}
 0 &=  \ave d\mu\left(\frac{\bar h}{r'_m}\right) \inf_{a\in \mathcal A_{k_m}}  \bigl( L_a u_{k_m}(\bar x+\bar h)   + c_a(\bar x+\bar h) \bigr) -0
 \\
&\le \inf_{a\in \mathcal A_{k_m}} \left(   \ave L_a u_{k_m}(\bar x+\bar h)  + c_a(\bar x+\bar h)  d\mu\left(\frac{\bar h}{r'_m}\right) \right) - \inf_{a\in \mathcal A_{k_m}} \bigl(   L_a u_{k_m}(\bar x)  + c_a(\bar x) \bigr).
\end{split}
\]

Recall that by \eqref{nontrinvP} we have $\sup_{a\in \mathcal A_k} \,[c_a]_{C^{\alpha}(B_1)}\le C_{0,k}$.
Hence, for all  $\bar x\in B_{3/4}(z)$ provided that $m$ is chosen large enough so that $r'_m {\rm diam}\,\bigl({\rm supp}\,\mu\bigr) \le 1/4$ we have
\begin{equation}\label{uuu}
\begin{split}
 - C_{0,k_m}  \ave d\mu\left(\frac{\bar h}{r'_m}\right)  |\bar h|^{^\alpha} &\le - \sup_{a\in \mathcal A_{k_m}} \, [c_a]_{C^\alpha(B_1)} \ave d\mu\left(\frac{\bar h}{r'_m}\right)  |\bar h|^{^\alpha}
 \\
  &\le \inf_{a\in \mathcal A_{k_m}} \left(   \ave L_a u_{k_m}(\bar x+\bar h) d\mu\left(\frac{\bar h}{r'_m}\right)  + c_a(\bar x)   \right)
 \\
 &\hspace{50mm}    - \inf_{a\in \mathcal A_{k_m}} \bigl(   L_a u_{k_m}(\bar x)  + c_a(\bar x) \bigr)
\\
&\le \sup_{a\in \mathcal A_{k_m}}  \left(   \ave L_a u_{k_m}k(\bar x+\bar h) d\mu\left(\frac{\bar h}{r'_m}\right)-   L_a u_{k_m}(\bar x)  \right)
\\
& \le M^+_{\mathcal L_0} \left(  \ave u_{k_m}(\,\cdot\,+\bar h) \,d\mu\left(\frac{\bar h}{r'_m}\right) -u_{k_m}\right)(\bar x).
\end{split}
\end{equation}

Note now that, since $\nu\le 2$,
\begin{equation}\label{difpols}
\delta^2 p(x+h,y)-\delta^2 p(x,y) = 0\quad \mbox{for all }p\in \mathcal P_\nu\mbox{  and for all }x,y,h\mbox{ in }\R^n.
\end{equation}
Taking into account \eqref{difpols}, we now translate \eqref{uuu} from $u_{k_m}$ to $v_{m}$. Using the definition of $v_m$ in \eqref{eqvm}, and setting $\bar h= r'_m h$ and $\bar x=z_m+ r'_mx$ in \eqref{uuu}, we obtain
\[\label{vv}
\begin{split}
- C_{0,k_m} (r'_m)^\alpha & \int_{\R^n} d\mu(h) |h|^\alpha  \le
\\
&\le  \frac{1}{(r'_m)^\sigma}M^+_{\mathcal L_0} \left(  (r'_m)^{\sigma+\alpha} \theta(r'_m)\left\{\ave v_{m}(\,\cdot\,+ h) \,d\mu(h) -v_{m}\right\}\right)(x)
\\
& \le  (r'_m)^{\alpha} \theta(r'_m) M^+_{\mathcal L_0} \left(  \ave v_{m}(\,\cdot\,+ h) \,d\mu(h) -v_{m}\right)(x)
\end{split}
\]
whenever  $|x| \le \frac{1}{C r'_m}$, and thus
\begin{equation}\label{vv}
- \frac {C_{0,k_m}}{\theta(r'_m)} C(\mu)  \le  M^+_{\mathcal L_0} \left(  \ave v_{m}(\,\cdot\,+ h) \,d\mu(h) -v_{m}\right) \quad \mbox{in  } |x| \le \frac{1}{C r'_m}.
\end{equation}
Given that $\mu$ has compact support, that $|C_{0,k_m}|\le 1$, and that $\theta(r'_m) \nearrow \infty$, we obtain that the left hand side of \eqref{vv} converges to zero $m\to +\infty$.
Thus, passing \eqref{vv} to the limit we find that
\[
0 \le M^+_{\mathcal L_0} \left(  \ave v(\,\cdot\,+ h) \,d\mu(h) -v\right) \quad \mbox{in all of }\R^n.
\]

Indeed, to carefully justify the previous limit $m \to +\infty$ on the right hand side of \eqref{vv} we are using that, by \eqref{growthc1}, the functions
\[w_m :=\ave v_{k_m}(\,\cdot\,+ h) \,d\mu(h) -v_{k_m}\]
satisfy, for all $R\ge {\rm diam}\,\bigl({\rm supp}\,\mu\bigr)$, that
\[  [w_m]_{C^{\sigma+\alpha'}(B_R)} \le CR^{\alpha-\alpha'}\quad \mbox{and}\quad \bigl|w_m(x)\bigr|\le C \int |h|^{\beta} d\mu(h) |x|^{\sigma+\alpha-\beta} \le C |x|^{\sigma+\alpha-\beta}.\]
Thus, taking $\beta\in \bigl(\alpha, \min\{1,\sigma+\alpha'\}\bigr)$, and  since $|x|^{\sigma+\alpha-\beta} \in L^1(\R^n, \omega_\sigma)$, we can use the dominated convergence theorem to compute the limit.
Therefore, the assumption (iii) of Theorem \ref{liouville} is satisfied by $v$.

A very similar (actually easier) computation shows that the assumption (ii) is also satisfied by $v$.
This finishes the proof the Claim.
\vspace{5pt}

We have thus proved that $v$ satisfies all the assumptions of Theorem \ref{liouville} and hence we conclude that $v$ is a polynomial of degree $\nu$.
On the other hand, passing \eqref{2} to the limit we obtain that $v$ is orthogonal to every polynomial of degree $\nu$ in $B_1$, and hence it must be $v\equiv 0$.
But then passing \eqref{nondeg35bis} to the limit we obtain that $v$ cannot be constantly zero in $B_1$; a contradiction.
\end{proof}

Using Proposition \ref{propmain} we prove an intermediate technical statement that will be later used to prove Theorem \ref{thm}.
\begin{prop} \label{corprop}
Let $\sigma\in (0,2)$, $\lambda$, $\Lambda$, $A_0$ and $C_0$ be given constants with $0<\lambda\le \Lambda$.
Suppose that $A_0\le 1$. There is $\bar \alpha>0$ (depending only on $n$, $\sigma$, $\lambda$ and $\Lambda$) such that the following statement holds.
Given $\alpha$ and $\alpha'$ satisfying $0<\alpha'<\alpha<\bar\alpha$ and  $\nu<\sigma+{\alpha'}<\sigma+\alpha$, where $\nu$ is the floor of $\sigma+\alpha$.
Let $u\in C^{\sigma+\alpha}(\overline{B_1})\cap C^\alpha(\R^n)$ be a solution of
\[ \cI(u,x)  = 0 \quad\mbox{ in }B_{1}, \]
where $\cI$, defined by \eqref{nontrinv}, satisfies \eqref{nontrinv1}, \eqref{nontrinv22}, and \eqref{nontrinv2}.

Then,
\begin{equation}\label{thegoal}
 [u]_{C^{\sigma+\alpha}(B_{1/4})}\le C\bigl(\|u\|_{C^{\sigma+{\alpha'}}(B_1)} + A_0  \|u\|_{C^{\sigma+\alpha}(B_1)} + \|u\|_{C^{\alpha}(\R^n)} + C_0 \bigl),
\end{equation}
where $C_0$ is the constant from \eqref{nontrinv2} and  $C$ depends only on $n$, $\sigma$, $\alpha$, ${\alpha'}$, $\lambda$,  $\Lambda$.
\end{prop}

\begin{proof}
Let   $\eta\in C^\infty_c(B_{1})$ be a cutoff function satisfying $\eta\equiv 1$ in $B_{3/4}$.
Then,
\begin{equation}\label{est-etau}
 \|\eta  u\| _{C^{\sigma+{\alpha'}}(\R^n)} \le C  \|u\|_{C^{\sigma+{\alpha'}}(B_1)}.
\end{equation}
In addition, we have
\begin{equation} \label{eqtruncated}
0 = \cI\bigl(\eta u + (1-\eta) u,x\bigr) =   \inf_{a\in \mathcal A} \left( \int_{\R^n} \delta^2 (\eta u) (x, y) K_a(0,  y)\,dy  + \tilde c_a(x) \right)
\end{equation}
where
\[
\tilde c_a(x) = c_{a}(x) + \boldsymbol A (x)+ \boldsymbol B(x),
\]
for
\[ \boldsymbol A(x) =  \int_{\R^n} \delta^2 (\eta  u) (x, y)  \bigl(K_a(x, y)- K_a(0,y) \bigr)\,dy,\]
and
\[ \boldsymbol B( x) = \int_{\R^n} \delta^2 \bigl((1-\eta) u\bigr) ( x, y)  K_a(x,  y)\,dy \]

We next write for $x, x'\in B_{1/2}$,
\[
\begin{split}
\boldsymbol A(x)-\boldsymbol A(x') &=  \int_{\R^n} \delta^2 (\eta  u) (x', y)  \bigl(K_a(x, y)- K_a(x',y) \bigr)\,dy\ +
\\
&\hspace{20pt} +  \int_{\R^n} \bigl(\delta^2 (\eta  u) (x, y)- \delta^2 (\eta  u) (x', y) \bigr)  \bigl(K_a(x, y)- K_a(0,y) \bigr)\,dy.
\\
&= \boldsymbol A_1(x,x') + \boldsymbol A_2(x,x')
\end{split}
\]

Let us now bound $|\boldsymbol A(x)-\boldsymbol A(x')|$.
We will do the case $\nu= \lfloor \sigma+\alpha\rfloor=2$ (the cases $\nu=1$ and $\nu=0$ are very similar).
On the one hand, we obtain
\[
\begin{split}
 \bigl|\boldsymbol A_1(x,x')\bigr| &\le  \int_{\R^n} |\delta^2 (\eta  u) (x', y) |\, \bigl|K_a(x, y)- K_a(x',y) \bigr|\,dy
\\
&\le \int_{B_{1/2}}  |y|^2 \|u\|_{C^{\sigma+\alpha}(B_1)} \, \bigl|K_a(x, y)- K_a(x',y) \bigr|\,dy\ +
\\
& \hspace{20mm}+ \int_{\R^n\setminus B_{1/2}}  \|u\|_{L^\infty(\R^n)}  \bigl|K_a(x, y)- K_a(x',y) \bigr|\,dy
\\
&\le C A_0 |x-x'|^\alpha \bigl(\|u\|_{C^{\sigma+\alpha}(B_1)} + \|u\|_{L^\infty(\R^n)} \bigr),
\end{split}
\]
where we have used  \eqref{nontrinv22}.

On the other hand, letting $d =|x-x'|$  we have
\[ \bigl|\delta^2 (\eta  u) (x, y)- \delta^2 (\eta  u) (x', y) \bigr| \le
\begin{cases}
 |y|^2 d^{\sigma+\alpha-2}  \|u\|_{C^{\sigma+\alpha}(B_1)} &\quad y \mbox{ in } B_d\\
 d^2 |y|^{\sigma+\alpha-2}  \|u\|_{C^{\sigma+\alpha}(B_1)} &\quad y \mbox{ in } B_{1/2}\setminus B_d\\
 d^{\alpha}  \|u\|_{C^{\alpha}(\R^n)} &\quad y \mbox{ outside } B_{1/2}.
 \end{cases}
 \]
Combining this and \eqref{nontrinv22} we readily obtain
\[
\begin{split}
 \bigl|\boldsymbol A_2(x,x')\bigr| &\le  \int_{\R^n} \bigl|\delta^2 (\eta  u) (x, y)- \delta^2 (\eta  u) (x', y) \bigr|\,  \bigl|K_a(x, y)- K_a(0,y) \bigr|\,dy
\\
&\le C  d^\alpha A_0  |x-0|^\alpha \bigl(\|u\|_{C^{\sigma+\alpha}(B_1)} + \|u\|_{C^\alpha(\R^n)} \bigr),
\end{split}
\]

Hence,
\begin{equation}\label{boldA}
[ \boldsymbol A]_{C^\alpha(B_{1/2})}  \le  C A_0  \bigl(\| u\|_{C^{\sigma+\alpha}(B_1)}+ \| u\|_{C^{\alpha}(\R^n)}\bigr).
\end{equation}

On the other hand, letting  $\tilde u = (1-\eta)u$ and using that $\tilde u \equiv 0$ in $B_{3/4}$, we obtain with similar computations
\[
\begin{split}
\bigl| \boldsymbol B( x) - \boldsymbol B( x') \bigr|
& \le  \int_{\R^n} \bigl| \delta^2 \tilde u  ( x, y)  - \delta^2 \tilde u ( x', y)  \bigr|  \,| K(x,y)|\,dy \ +
\\
&\hspace{25mm}+ \int_{\R^n}  | \delta^2\tilde u ( x, y)|     \,\bigl| K_a(x,  y)- K_a(x',  y)\bigr| \,dy
\\
&\le C(\Lambda+ A_0) |x-x'|^\alpha \|u\|_{C^\alpha(\R^n)}.
\end{split}
\]
Hence,
\begin{equation}\label{boldB}
[ \boldsymbol B]_{C^\alpha(B_{1/2})}  \le C(\Lambda+ A_0) \|u\|_{C^\alpha(\R^n)}.
\end{equation}

Therefore, using \eqref{boldA} and \eqref{boldB}, and recalling that we assume that $A_0 \le 1$,  we obtain
\begin{equation}\label{tildecdex}
\bigl[ \tilde c_a(x)\bigr]_{C^\alpha(B_{1/2})} \le  C_0  + C A_0 \|u\|_{C^{\sigma+\alpha}(B_1)} + C\| u\|_{C^\alpha(\R^n)}
\end{equation}
where $C$ depends only on $n$, $\sigma$, $\lambda$ and $\Lambda$.

We have thus proven that  the function $\eta u$ belongs to  $C^{\sigma+{\alpha'}}(\R^n)$ with the control \eqref{est-etau} on this norm  and solves the equation \eqref{eqtruncated} in $B_{1/2}$ with $\tilde c(x)$ satisfying \eqref{tildecdex}. Hence, $\eta u$ satisfies the assumptions  of Proposition \ref{propmain}
and therefore \eqref{thegoal} follows from the estimate provided by the same proposition.
\end{proof}

As a last ingredient for the proof of Theorem \ref{thm}, we recall the adimensional H\"older seminorms  from the classical book Gilbarg-Trudinger \cite{GT}.
We next recall the definition of the adimensional $C^\beta$ seminorm from Section 4 of \cite{GT}. Let $\beta> 0$ and let $k$ be the integer such that $\beta =k+\beta'$ for some $\beta'\in (0,1]$. Then,
\[  [u]_{\beta;\Omega} ^* = \sup_{x,y\in \Omega, |l|=k}  (d_{x,y})^{\beta} \frac{|D^l u (x) - D^lu(y)|}{|x-y|^{\beta'}}, \]
where $d_{x,y}:= \min \{ {\rm dist}(x,\partial\Omega), {\rm dist}(y, \partial \Omega)\}$.

We next give the

\begin{proof}[Proof of Theorem \ref{thm}]
Let $\rho \in (0,1)$  and $z\in B_1$ be such that $B_\rho(z)\subset B_1$. Let $\bar u(\bar x) = u (z+\rho\bar x)$.
The function $\bar u$ solves in $B_1$ the rescaled equation
\begin{equation}\label{rescaledeqn}
\bar \cI(\bar u,\bar x)= \min_a \left( \int_{\R^n} \delta^2 \bar u(\bar x,\bar y) \rho^{n+\sigma}K_a(z+\rho\bar x,\rho\bar y)\, d\bar y  + \rho^\sigma c_a(z+\rho\bar x) \right) =0
\end{equation}
in $B_1$.
Note that if the kernels $K_a(x,y)$ of the original operator $\cI$ satisfy  \eqref{nontrinv1}-\eqref{nontrinv22}-\eqref{nontrinv2}, then the rescaled  kernels
\[\bar K_a(\bar x,\bar y) :=\rho^{n+\sigma}K_a(z+\rho\bar x,\rho\bar y)\]
of $\bar\cI$ also satisfies  \eqref{nontrinv1}-\eqref{nontrinv22}-\eqref{nontrinv2}  with the same constants $\lambda$, $\Lambda$, $A_0$, $C_0$ as those of $\cI$.
In fact, we have
\[\begin{split}
\int_{B_{2r}\setminus B_r} \bigl| \bar K_a(\bar x,\bar y)- \bar K_a(\bar x',\bar y)\bigr| d\bar y &=
\\
& \hspace{-40pt}=  \int_{B_{2\rho r}\setminus B_{\rho r} }
 \rho^{n+\sigma} \bigl| K_a(z+\rho\bar x, y)- K_a(z+\rho\bar x,  y)\bigr|  \frac{dy}{\rho^n}
\\
& \hspace{-40pt} \le A_0 |\rho(\bar x- \bar x')|^\alpha \frac{2-\sigma}{r^\sigma}
\end{split}
\]
Hence, as it will be used on in this proof, $\bar \cI$ satisfies \eqref{nontrinv22} with $A_0$ replaced by $\rho^\alpha A_0\le A_0$.

Let $\nu= \lfloor\sigma+\alpha\rfloor$ and $\alpha' = \alpha'(\sigma, \alpha,\nu)$ be given by \eqref{alpha'}. Since $\sigma+\alpha>\nu$ by assumption (it is not an integer) we have $\alpha'\in(0,\alpha)$ and $\nu<\sigma+\alpha'$.
Then,  assuming that $A_0\le1$, Proposition \ref{corprop} applied to $\bar u$ yields
\begin{equation}\label{thegoal2}
 [\bar u]_{C^{\sigma+\alpha}(B_{1/4})}\le C\bigl(\|\bar u\|_{C^{\sigma+{\alpha'}}(B_1)} + A_0 \|\bar u\|_{C^{\sigma+\alpha}(B_1)} +  \|u\|_{C^{\alpha}(\R^n)} + C_0 \bigl),
\end{equation}
where $C_0$ is the constant from \eqref{nontrinv2} and  $C$ depends only on $n$, $\sigma$, $\alpha$,  $\lambda$,  $\Lambda$.

Using standard  interpolation inequalities in $B_1$ to control the full  norm $\|\cdot\|_{C^{\sigma+\alpha}(B_1)}$ by  $[\,\cdot\,]_{C^{\sigma+\alpha}(B_1)} +\|\cdot\|_{L^\infty(B_1)}$,   and scaling back \eqref{thegoal2} from $\bar u$ to $u$,  we obtain
\[
\rho^{\sigma+\alpha}[u]_{C^{\sigma+\alpha}(B_{\rho/4(z)})}\le C\bigl( \rho^{\sigma+{\alpha'}}[u]_{C^{\sigma+{\alpha'}}(B_\rho(z))} +  A_0 \rho^{\sigma+\alpha}[u]_{C^{\sigma+\alpha}(B_{\rho}(z))}
+\| u \|_{C^\alpha(\R^n)} + C_0 \bigr).
\]

The previous estimate holds in every ball $B_\rho(z)\subset B_1$ and this immediately yields, in terms of the adimensional H\"older norms,  that
\[ [u]^*_{\sigma+\alpha; B_1} \le C\bigl( [u]^*_{\sigma+{\alpha'};B_1} + A_0 [u]^*_{\sigma+\alpha;B_1} +\| u \|_{C^\alpha(\R^n)} + C_0\bigr).\]
Then, assuming that  $A_0\le \varepsilon_0$, with $\varepsilon_0$ small enough (depending only on $n$, $\sigma$, $\lambda$, and $\Lambda$), and using the interpolation inequality for adimensional H\"older norms \cite[Lemma 6.32 in Section 6.8]{GT}
\[ [u]^*_{\sigma+{\alpha'}; B_1} \le \delta [u]^*_{\sigma+\alpha; B_1} + C(\delta) \|u\|_{L^\infty(B_1)}\]
we obtain
\[  [u]^*_{\sigma+\alpha; B_1} \le C(\| u \|_{C^\alpha(\R^n)}+ C_0).\]
Since clearly $[u] _{C^{\sigma+\alpha}(B_{1/2})} \le C[u]^*_{C^{\sigma+\alpha}(B_1)}$  we have proven the theorem in the case of $A_0 \le \varepsilon_0 \ll 1$.

Let us prove now the Theorem also for non-small  $A_0$.
We only need to use a typical scaling trick.
Let as before $z\in B_{1/2}$ and $\rho\in (0,1]$ such that  $B_\rho(z)\subset B_1$.
We have already seen that if the function $\bar u = u(z+\rho\bar x)$ solves the rescaled equation \eqref{rescaledeqn} in $B_1$ and that $A_0$  in the new equation by $\rho^\alpha A_0$.
Therefore, choosing $\rho$ small enough ---depending on $A_0$ and $\alpha$--- so that $\rho^\alpha A_0 \le \varepsilon_0$ we may apply the previous estimate to the rescaled equation to obtain
\[ [\bar u] _{C^{\sigma+\alpha}(B_{1/2})} \le C(\|\bar u \|_{C^\alpha(\R^n)}+ C_0) \]
and thus
\[ [u] _{C^{\sigma+\alpha}(B_{\rho/2}(z))} \le C\rho^{-\sigma-\alpha}(\|\bar u \|_{C^\alpha(\R^n)} + C_0 ).\]
Since a finite number of these balls cover $\overline {B_{1/2}}$, the estimate of the Theorem follows.
\end{proof}

To end the section we give the

\begin{proof}[Proof of Corollary \ref{Lalpha}]
First, note that using for instance the H\"older estimate in \cite{CS},  the solution $u$ belongs to $C^\alpha(\overline {B_{7/8}})$ with and estimate ---note that $\alpha<\bar\alpha$ with $\bar \alpha$ small enough.
Let $\eta\in C^\infty_c(B_1)$ be a smooth cutoff function with $\eta\equiv 1$ in $B_{6/8}$ and $\eta \equiv0$ outside $B_{7/8}$.

Then, $\eta u \in C^\alpha(\R^n)$ solves the following equation in $B_{5/8}$:
\begin{equation} \label{eqtruncated}
0 = \cI\bigl(\eta u + (1-\eta) u,x\bigr) =   \inf_{a\in \mathcal A} \left( \int_{\R^n} \delta^2 (\eta u) (x, y) K_a(x,  y)\,dy  + \tilde c_a(x) \right)
\end{equation}
where
\[
\tilde c_a(x) = c_{a}(x) + \int_{\R^n} \delta^2 \bigl((1-\eta) u\bigr) ( x, y)  K_a(x,  y)\,dy.
\]
Using the additional assumption \eqref{nontrinv3} and the fact that  $(1-\eta)\equiv 0$ in $B_{6/8}$ we readily show that $\|\tilde c_a\|_{C^\alpha(B_{5/8})}\le C_0 +C \|u\|_{L^\infty(\R^n)}$.
Therefore, under the assumptions of the Corollary, the function $\eta u \in C^\alpha(\R^n)$ solves an equation in $B_{5/8}$ that satisfies the assumptions of Theorem \ref{thm} with $B_1$ replaced by $B_{5/8}$.
Then, applying the estimate of Theorem \ref{thm} (rescaled) to the function  $\eta u$ the Corollary follows ---since the equation is satisfied in $B_{5/8}$ instead of $B_1$ we need to use the standard covering argument to obtain the estimate in $B_{1/2}$ instead of $B_{5/16}$.
\end{proof}

\section{Approximation procedure for non translation invariant equations}
\label{sec:approx}

In this section we show a way of approximating a non translation invariant equation $\cI(u,x)=0$ in $B_1$ of the form  \eqref{nontrinv} and satisfying \eqref{nontrinv1}-\eqref{nontrinv22}-\eqref{nontrinv2} by a sequence of equations
$\cI^\epsilon(u^\epsilon,x)=0$ that admit $C^3$ solutions in $B_1$. This approximation procedure is modification of the one in \cite{CS3}.

For $\epsilon\in(0,1)$, let
\begin{equation}\label{Bnontrinv}
\cI^\epsilon (u^\epsilon,x) :=  \inf_{a\in \mathcal A} \left( \int_\R \delta^2 u^\epsilon(x,y) K^\epsilon_a(x,y)\,dy  + c^\epsilon_a(x) \right)
\end{equation}
where, for all $a\in \mathcal A$ and for all $x$ in $B_1$ and $y \in \R^n\setminus \{0\}$, we have
\begin{equation}\label{Kepsilon}
\begin{split}
 K_a^ \epsilon (x,y) &=
 \xi\left(\frac{y}{4\epsilon}\right)\frac{(2-\sigma)}{|y|^{n+\sigma}}\ +
   \\
 & \hspace{10mm}+\left(1-\xi\left(\frac{y}{4\epsilon}\right) \right) \int_{\R^n} \frac{d\bar x}{\epsilon^n}\int_{\R^n}\frac{d\bar y}{\epsilon^n} K_a(x-\bar x, y-\bar y)\, \eta\left(\frac{\bar x}{\epsilon}\right)\eta\left(\frac{\bar y}{\epsilon}\right),
 \end{split}
\end{equation}
and
\begin{equation}\label{cepsilon}
 c_a^ \epsilon (x) = \int_{\R^n} \frac{d\bar x}{\epsilon^n} c_a(x-\bar x)\, \eta\left(\frac{\bar x}{\epsilon}\right),
\end{equation}
for some $\xi\in C^\infty_c(B_1)$ with $\xi\equiv 1$ in $B_{1/2}$ and for some $\eta\in C^\infty_c(B_1)$ with $\eta\ge0$ and $\int\eta =1$.

\begin{rem}\label{cttsofapprox}
Note that the operator $\cI^\epsilon$  satisfies \eqref{nontrinv1}-\eqref{nontrinv22}-\eqref{nontrinv2} ---as $\cI$--- with the same constants $A_0$, $C_0$, and with $\lambda$, $\Lambda$  replaced by $\lambda/C$, $C\Lambda$ respectively. If in addition the operator
$\cI$ satisfies \eqref{nontrinv3} then so does  $\cI^\epsilon$ again with $\Lambda$ being replaced by  $C\Lambda$.
Note in addition that $\cI^\epsilon \rightarrow \cI$ in weakly in $B_1$ and with the weight $\omega_\sigma$---for the notion of weak convergence of nonlocal elliptic operators see \cite{CS2}.
\end{rem}

We will prove the following
\begin{prop}\label{propapproximat}
For all $\epsilon>0$, the Dirichlet problem
\begin{equation}\label{dirichletapprox}
\begin{cases}
\cI^\epsilon(u^\epsilon, x)=0 \quad &\mbox{in  }B_1\\
u= g & \mbox{outside }B_1
\end{cases}
\end{equation}
with bounded $g\in C(\R^n\setminus B_1)$ admits a unique solution $u^\epsilon \in C(\R^n)\cap C^3(B_1)$.
\end{prop}

\begin{proof}
To show that, for all $\epsilon>0$ the Dirichlet problem \eqref{dirichletapprox} admits a  unique solution $u\in C(\R^n)\cap C^3(B_1)$ we will use Perron's method.
Since a comparison principle between viscosity solutions is not available for non translation invariant nonlocal fully nonlinear equations,  the use of Perron's method will be based in  the following crucial observation
(existence of smooth solutions in tiny balls for the regularized equation).

\vspace{5pt}
\noindent { \bf Claim.}{ \em
Given $\epsilon>0$, there is $\delta_0>0$ with $\delta_0 \ll \epsilon$ such that whenever $B_{\delta}(z)$ is a ball contained in $B_1$ with $\delta\in (0,\delta_0)$  there exists a unique solution $w\in C(\R^n)\cap C^3(B_\delta)$ to the Dirichlet problem
\begin{equation}\label{dirichletapprox}
\begin{cases}
\cI^\epsilon(w, x)=0 \quad &\mbox{in  }B_\delta(z)\\
w= h & \mbox{in  }\R^n \setminus B_\delta(z)
\end{cases}
\end{equation}
for all continuous complement data $h$ with  $\|h\|_{ L^\infty(\R^n\setminus B_\delta)}\le 1$.

Moreover, the function $w$ satisfies
\begin{equation}\label{estimatew}
\|w\|_{C^3(B_\delta(z))} \le C \|h\|_{L^\infty(\R^n)}
\end{equation}
where $C$ depends on $n$, $\lambda$, $\Lambda$, $\epsilon$, and $\delta$.
}
\vspace{5pt}

To prove the Claim, for fixed $\epsilon>0$ we rescale the operator $\cI^\epsilon$ as follows
\[ \cI^\epsilon(w,z+ \delta \bar x) = \delta^{-\sigma} \bar \cI(w(z+\delta \,\cdot\,), \bar x).\]
Note that the kernels that define the new operator $\bar \cI$ are smooth and coincide to that of the fractional  Laplacian inside of a large ball $B_{\epsilon/\delta}$ (recall that $\delta \ll \epsilon$). Hence,
writing $C \bar \cI(\bar w,\bar x) = (-\Delta)^{\sigma/2} \bar w (x)+ \mathcal N_\delta (\bar w, \bar x)$ for the rescaled function $\bar w = w(z+\delta \,\cdot\,)$ the problem \eqref{dirichletapprox} takes the form
\begin{equation}\label{dirichletapproxrescaled}
\begin{cases}
-(-\Delta)^{\sigma/2} \bar w +   \mathcal N_\delta(\bar w, \bar x) =0\quad &\mbox{in  }B_1\\
\bar w= \bar h & \mbox{in  }\R^n \setminus B_1,
\end{cases}
\end{equation}
where $x = z+\delta \bar x$,
\[\mathcal N_\delta(\bar w, \bar x) =  C \inf_{a\in \mathcal A} \left( \int_\R \delta^2 u^\epsilon(x, \bar y) \bigl( \delta^{n+\sigma} K^\epsilon_a(\bar x,\delta y) - (2-\sigma)|y|^{-n-\sigma}\bigr)\,dy  + \delta^{\sigma} c^\epsilon_a(x) \right) \]
and  $\bar h(\bar x) = h(z+\delta\bar x)$.

Notice that
 \[\delta^{n+\sigma} K^\epsilon_a(\bar x,\delta y) - (2-\sigma)|y|^{-n-\sigma}\equiv 0\quad \mbox{in }B_{2\epsilon/\delta}\]
 by the definition of $K^\epsilon_a$ in \eqref{Kepsilon}. Then, it is straightforward to verify ---using \eqref{Kepsilon} and \eqref{cepsilon}---that
\begin{equation}\label{theN}
\| \mathcal N_\delta(w, \,\cdot\,)\|_{C^3(\overline{B_1})} \le \gamma_{\delta} \|w\|_{L^\infty(\R^n)}
\end{equation}
and
\begin{equation}\label{theN2}
\| \mathcal N_\delta(w, x) - \mathcal N_\delta(w', x) \|_{L^\infty({B_1})} \le \gamma_{\delta} \|w-w'\|_{L^\infty(\R^n)}
\end{equation}
for every $w, w'\in L^\infty(\R^n)$ where
\[ \gamma_\delta \searrow 0  \quad \mbox{as}\quad (\delta/\epsilon) \searrow 0.\]

Therefore, a solution $w\in L^\infty(\R^n)\cap C^3(B_1)$ to \eqref{dirichletapproxrescaled} can be then constructed using the solvability of the Dirichlet problem with the fractional Laplacian and  the Banach fixed point theorem.
Indeed, let \[\mathcal I_{\bar h} [f] =: v\] be the unique solution to
\begin{equation}\label{blabla}
\begin{cases}
(-\Delta)^{\sigma/2}  v  = f  \quad &\mbox{in  }B_1\\
v = \bar h & \mbox{in  }\R^n \setminus B_1,
\end{cases}
\end{equation}
Then,  \eqref{dirichletapproxrescaled} can be restated as a fixed point problem as
\[\mathcal I_{\bar h} \bigl[ \mathcal N_\delta(w, \,\cdot\,) \bigr] = w.\]
The contractivity of the previous map in the ``closed ball'' $\{ w \,: \, \|w\|_{L^\infty(B_1)} \le 2\}$ when $\delta/\epsilon \ll 1$ follows form \eqref{theN}-\eqref{theN2} and  the elementary estimate for \eqref{blabla}
\[ \|v\|_{L^\infty(B_1)} \le \|\bar h\|_{L^\infty(\R^n \setminus B_1)} + C\|f\|_{L^\infty(B_1)}.\]

The continuity up to the boundary of $w$ ---with implies the uniqueness of solution to \eqref{dirichletapprox} in the class of viscosity solutions--- follows from the results in  the Section 3 of \cite{CS2}.
This finishes the proof of the Claim.
\vspace{5pt}

The previous Claim makes now it simple to apply of Perron's method to show existence of solution. As usual,  we consider the following candidate to viscosity solution to \eqref{dirichletapprox}:
\begin{equation}\label{Perron}
 u^\epsilon(x) = \sup \bigl\{ w(x) \,: \, w\in C(\overline \Omega) \mbox{ is a viscosity subsolution of }\eqref{dirichletapprox} \bigr\}.
\end{equation}
Using the Claim, and the barriers from Section 3 in \cite{CS2}, the ideas of the classical proof by Perron's method of the existence of a harmonic function with given continuous boundary data in smooth domains apply to this case, since we also have solvability in balls (in our case tiny ones).
We obtain that $u^\epsilon$ solves classically the equation in the interior and attains continuously the complement data.

Indeed, as for harmonic functions, in the supremum of \eqref{Perron}  defining $u(x)$, for every $\delta\in(0,\delta_0)$ such that  $B_\delta(z)\subset B_1$ we can replace the subsolution $w$ by the solution in $B_\delta(z)$ with its  same values outside.
The new function will be larger by the comparison principle between a viscosity an a smooth solutions.
It then follows using \eqref{estimatew} and Arzel\`a-Ascoli that $u^\epsilon$ belongs $C^3(B_1)$, and that it is a solution to the equation in the interior of $B_1$.
That $u^\epsilon$ defined as in \eqref{Perron} is  continuous  function up to the boundary attaining the complement data follows from standard barrier arguments, employing the barriers from Section 3 of \cite{CS}.
\end{proof}

The remaining part of this section will be devoted to the proof of Theorem \eqref{thm2}. In it, we will need the following Proposition.

\begin{prop}\label{lemcosa}
Let $\sigma\in (0,2)$, and $\lambda$, $\Lambda$ be given constants with $0<\lambda\le \Lambda$.
Then, there exists $\gamma\in (0,1)$ depending only on $n$, $\sigma$, $\lambda$, $\Lambda$ such that the following statement holds.

Let $\alpha\in (0,\gamma)$ and assume that $u\in C^{\sigma + \alpha}(B_1)\cap C(\R^n)$ is a solution to
\begin{equation}\label{cbonet}
\begin{cases}
M^+_{\mathcal L_0} u \ge -C_0 \quad & \mbox{in } B_1\\
M^+_{\mathcal L_0} u \le C_0  &\mbox{in } B_1\\
u=g &\mbox{in } \R^n \setminus B_1,
\end{cases}
\end{equation}
with $g\in C^\alpha(\R^n\setminus B_1)$.
Then, $u \in C^\alpha(\R^n)$ with the estimate
\begin{equation}\label{estalphaalpha}
 \|u\|_{C^\alpha(\R^n)}\le C\bigl( \|g\|_{C^\alpha(\R^n \setminus B_1)}+C_0\bigr),
\end{equation}
where $C$ depends only on $n$, $\sigma$, $\lambda$, $\Lambda$, and $\alpha$.
\end{prop}

\begin{rem}
The only ``novelty'' of the previous proposition with respect to the results in \cite{CS2} is that there is no loss in the exponent: from a $C^\alpha$ exterior data we obtain  $C^\alpha$ regularity  up to the boundary  (the same $\alpha$).
Note that in the proposition  $\gamma$ is small and $\alpha< \gamma$.
Even for a linear translation invariant equations such a result is not true for all $\alpha$.
Indeed, even for the equation $\Delta u=0$ in $B_1$,
it is well-known that Lipschitz boundary data may lead to a non-Lipschitz harmonic extension.
The exponents $1,2,3,...$ (Lipschitz, $C^{1,1}$, $C^{2,1}$, ...) are in some sense critical for the boundary regularity of $\Delta$ because there exist harmonic polynomials that are degree $1,2,3, ...$  and that solve $\Delta u=0$ in $\R^n_+$ and $u=0$ on $\{x_n=0\}$.
Related to this, it is worth it to point out that a small modification of the proof of Propostion \ref{lemcosa} shows that that solutions to $(-\Delta)^s =0$ in $B_1$ with $C^\alpha$ exterior data are $C^\alpha$ up to the boundary whenever $\alpha<s$.
However, we do not expect the result to be true for $\alpha=s$. Again, the criticality of the exponent $s$ comes from the fact that $(x_n)_+^s$ is a solution to the fractional Laplacian equation in the half space.
\end{rem}

Proposition \ref{lemcosa} will follow from an easy blow-up and compactness argument and from the following Liouville type result
\begin{lem} \label{lemliouville}
Let $\sigma$,  $\lambda$, $\Lambda$, $\gamma$, $\alpha$, as in the statement of Proposition \ref{lemcosa}.

Assume that $u\in C(\R^n)$ is a viscosity solution to
\[
\begin{cases}
M^+_{\mathcal L_0} u \ge 0 \quad & \mbox{in } H\\
M^+_{\mathcal L_0} u \le 0  &\mbox{in } H\\
u=0 &\mbox{in } \R^n \setminus H,
\end{cases}
\]
where $H$ is the whole $\R^n$ or some half space, and assume that $u$ satisfies the growth control
\[\|u\|_{L^\infty(B_R)}\le CR^\gamma\]
for all $R\ge 1$.

Then, $u$ is constant ($u\equiv 0$ when $H\neq\R^n$).
\end{lem}
\begin{proof}
Let $\rho \ge 1$ and $\bar u(x) = \rho^{-\alpha} u(\rho x)$. Letting $\bar H = H/\rho$, we have that $\bar u$ solves
\[
\begin{cases}
M^+_{\mathcal L_0} \bar u \ge 0 \quad & \mbox{in } \bar H\cap B_1\\
M^+_{\mathcal L_0} \bar u \le 0 &\mbox{in } \bar H \cap B_1\\
\bar u=0 &\mbox{in } B_1 \setminus \bar   H.
\end{cases}
\]
In addition $\bar u$ satisfies the growth control $\|\bar u\|_{L^\infty(B_R)} = \|\rho^{-\alpha} u\|_{L^\infty(B_{\rho R})} \le CR^{\alpha}$.
Thus, in particular $|u|\le C$  in $B_1$ and $\int_{\R^n} |u(y)| \bigl(1+|y|\bigr)^{-n-\sigma}dy \le C$.

Therefore if follows, using the interior and boundary regularity results from \cite{CS} and \cite{CS2}, that
\begin{equation}\label{estest}
 \|\bar u \|_{C^\gamma(B_{1/16})} \le C,
\end{equation}
for some small $\gamma$ depending only on $n$, $\sigma$, $\lambda$, and $\Lambda$.

Let us next give the details of the proof of \eqref{estest}. There are only two nontrivial cases: that $\bar H$ contains $B_{1/8}$, or that $\partial \bar H$ has nonempty intersection with  $B_{1/8}$.
Otherwise  $\bar H\cap B_{1/4}= \varnothing$ and  \eqref{estest} is trivial since  $u\equiv 0$ in $B_{1/8}$.

In the first case ($B_{1/8}\subset \bar H$),  \eqref{estest} follows from the interior estimates in \cite{CS2}.

In the second case, there will be some point $z$ in the intersection $\partial \bar H\cap B_{1/8}$, and $\bar u$ solves an equation in half of $B_{1/2}(z)$ and vanishes in the
complementary half ball.
Then, a barrier argument shows that, for some small $p>0$,
\begin{equation}\label{aaaaa}
  |\bar u| \le C {\rm dist}\bigl(x, \R^n\setminus \bar H)^p \quad \mbox{in }B_{1/4}(z).
\end{equation}
Indeed,  the function $\psi(x)={\rm dist}\,\bigl(x, B_{1/10}\bigr)^p$ is, for $p$ small enough, a supersolution in the annulus $B_{1/10+\epsilon}\setminus B_{1/10}$,
for some $\epsilon>0$.  Namely, it satisfies $M^+_{\mathcal L_0} \psi \le 0$ there ---see for instance \cite[Lemma 3.3]{CS2}.
Using translates of $C\psi$ (respectively $-C\psi$) as upper  (lower) barrier we readily show \eqref{aaaaa}.
Combining it with the interior estimates ---this is standard, see for instance the proof of Theorem 3.3 in \cite{CS2}---  we obtain
\[ \|\bar u\|_{C^\gamma(B_{1/4}(z))} \le C\]
for some $\gamma>0$ (smaller than $p$ and than the exponent of interior regularity).
Then \eqref{estest} follows since clearly $B_{1/16}\subset B_{1/4}(z)$ ---recall that $z\in B_{1/8}$.

Finally we scale back \eqref{estest} from $\bar u$ to $u$ and we obtain  that, for all $\rho\ge 1$,
\[  [u ]_{C^\gamma(B_{\rho/16})} \le C \rho^{\alpha-\gamma}.\]
Sending $\rho\nearrow +\infty$ we obtain that  $[u]_{C^\gamma(\R^n)}=0$ and thus $u$ is constant.
\end{proof}

Let us now give the
\begin{proof}[Proof of Proposition \ref{lemcosa}]
Since $u$ is a solution of \eqref{cbonet}  then by  \cite[Theorem 3.3]{CS2} that $\|u\|$ satisfies the estimate
\begin{equation}\label{estalphaepsilon}
\|u\|_{C^{\alpha'}(B_1)}\le C\bigl(\|g\|_{C^\alpha(\R^n \setminus B_1)} + C_0\bigr)
\end{equation}
for some ${\alpha'}>0$ and $C$ depending only on $n$, $\sigma$, $\lambda$, $\Lambda$.
Note that although Theorem 3.3 in \cite{CS2} is stated with a general modulus of continuity, a inspection of to its proof shows that
a H\"older modulus of continuity for the exterior datum leads to another (worse) H\"older modulus of continuity up to the boundary.

By homogeneity we may always assume that $\|g\|_{C^\alpha(\R^n\setminus B_1)}+ C_0= 1$.

We want to show that the previous estimate \eqref{estalphaepsilon} holds also with ${\alpha'}$ replaced by $\alpha$, provided that $\alpha\in(0,\gamma)$, where $\gamma$ is the
exponent from  Lemma \ref{lemliouville}. That is, we want to establish \eqref{estalphaalpha}. The proof is by contradiction.

Similarly as in the proof of Proposition \ref{propmain}, if the estimate \eqref{estalphaalpha} is false
 then, for each integer $k\ge 0$, there exists  $g_k$, $C_{0,k}$,  and $u_k$, satisfying \eqref{cbonet} ---with $u$ and $g$ replaced by $u_k$  and $g_k$ respectively---
such that
\[ \|u_k\|_{C^{\alpha}(B_{1})} \ge k,\]
while $\|g_k\|_{C^\alpha(\R^n\setminus B_1)}+ C_{0,k}= 1$.

Using Lemma \ref{lemaux} we then have
\begin{equation}\label{2k22}
\sup_k \sup_ {z\in B_{1}} \sup_{r>0} \ r^{{{\alpha'}}-\alpha}\left[ u_k \right]_{C^{{\alpha'}}(B_{r}(z))} = +\infty,
\end{equation}

Next we define
\[
 \theta(r) := \sup_k  \sup_ {z\in  B_{1/2}}  \sup_{r'>r}  (r')^{{{\alpha'}}-\alpha}\,\bigl[u_k\bigr]_{C^{{\alpha'}} \left(B_{r'}(z)\right)} \,.
\]
Note that $\theta$ is monotone nonincreasing and $\theta(r)<+\infty$ for $r>0$ since we are assuming that  $\|g_k\|_{C^\alpha(\R^n\setminus B_1)}+ C_{0,k}= 1$ and hence by \eqref{estalphaepsilon} we have $\|u_k\|_{C^{\alpha'}(\R^n)}\le C$.
In addition, by \eqref{2k22} we have  $\theta(r)\nearrow +\infty$ as $r\searrow0$.

As in the proof of Proposition \ref{propmain}, there are sequences $r_m\searrow 0$, $k_m$, and $z_{m}\to z \in \overline B_{1/2}$, for which
\begin{equation}\label{nondeg222}
(r_m')^{{{\alpha'}}-\alpha}  \bigl[u_{k_m}\bigr]_{C^{{\alpha'}} \left(B_{r_m'}(z_m)\right)} \ge \frac{1}{2} \theta(r_m').
\end{equation}

We then consider the blow-up sequence
\[ v_m(x) = \frac{u_{k_m}(z_{m} +r_m' x)-u_{k_m}(0)}{(r_m')^{\alpha}\theta(r_m)}.\]

Note also that \eqref{nondeg222} is equivalent to the following inequality for all $m\geq1$:
\begin{equation}\label{nondeg3522}
[v_m]_{C^{{\alpha'}}(B_1)}\ge 1/2,
\end{equation}

Similarly as in the proof of Proposition \ref{propmain} we  obtain
\begin{equation}\label{growthc022}
[v_{m}]_{C^{{\alpha'}} (B_R)} \leq CR^{\alpha-{\alpha'}}\quad \textrm{for all}\ \,R\ge 1.
\end{equation}
and
\begin{equation}\label{growthc122}
\|v_{m}\|_{L^{\infty} (B_R)} \leq CR^{\alpha}\quad \textrm{for all}\ \,R\ge 1.
\end{equation}

As in the proof of Proposition \eqref{propmain}, by further rescaling $v_m$ if necessary, we may assume that in addition to \eqref{nondeg3522} the following holds
\begin{equation}\label{nondeg3522bis}
{\rm osc}_{B_1}  v_m  \ge 1/4.
\end{equation}

Using  \eqref{growthc022}, \eqref{growthc122} and the stability results for viscosity supersolutions and subsolutions  \cite[Lemma 4.3]{CS2} we
obtain that a subsequence of $v_m$ converges locally uniformly in $\R^n$ to a function $v$ satisfying the assumptions of Lemma \ref{lemliouville}.
Hence, $v$ is constant. Since  $v_m(0)=0$ for all $m$ we must have $v\equiv 0$, but then we reach a contradiction passing  \eqref{nondeg3522bis} to the limit.
\end{proof}

We finally give the

\begin{proof}[Proof of Theorem \ref{thm2}]
Let $u^\epsilon$ be the solution to \eqref{dirichletapprox}, whose existence is guaranteed by Proposition \ref{propapproximat}.

Since $\cI$ is an operator of the form \eqref{nontrinv} satisfying \eqref{nontrinv1}-\eqref{nontrinv22}-\eqref{nontrinv2} then so is the regularized operator $\cI^\epsilon$ up to replacing  $\lambda$, $\Lambda$  by $\lambda/C$, $C\Lambda$---see Remark \ref{cttsofapprox}.
Note that $M^+_{\mathcal L_0} u^\epsilon \ge - \sup_{a\in A} \|c_a\|_{L^\infty(B_1)}  \ge -C_0$ in $B_1$ and similarly $M^{-} u^\epsilon  \le  C_0$ in $B_1$.
Then,  Theorem 3.3 in \cite{CS2} provides with a modulus of continuity in $\overline {B_1}$ for  $u^\epsilon$ --- this modulus of continuity depends on the modulus of continuity of $g$, $n$, $\sigma$, $\lambda$, $\Lambda$, and $C_0$, but not on $\epsilon$.
Therefore, using the Ascoli-Arzel\`a theorem, is a sequence $\epsilon_m\searrow 0$ and a function $u\in C^0(\overline{B_1})$ such that $u^{\epsilon_m}\rightarrow u$  uniformly in $\overline B_1$.
Since  $\cI^\epsilon \rightarrow \cI$ weakly as $\epsilon \searrow 0$, it follows from the``stability lemma'' \cite[Lemma 4.3]{CS2} that the limiting function $u$
is a viscosity solution of $\cI (u,x)=0$ in $B_1$  that attains continuously the  complement data $g$.

Let us prove that in both cases (a) and (b) the viscosity  solution $u$ belongs to $C^{\sigma+\alpha}(B_1)$ and hence it is a classical solution.
For any $z\in B_1$ and $\rho>0$ such that $B_{\rho} \subset B_1$ consider the resealed function $\bar u^\epsilon(\bar x) = u^\epsilon (z+\rho \bar x)$. Exactly as in the proof of Theorem \eqref{thm}, the function $\bar u$ satisfies in
$B_1$ the rescaled equation $\bar \cI^\epsilon(\bar u,\bar x)=0$, where $\bar \cI^\epsilon$ is still of the form  \eqref{nontrinv}-\eqref{nontrinv2}-\eqref{nontrinv3} with the same $C_0$ , $A_0$ and ellipticity constants as $\cI^\epsilon$.

In the case (a), using Proposition \ref{lemcosa} we find that $u^\epsilon \in C^\alpha(\R^n)$ with
\[\|u^\epsilon \|_{C^\alpha(\R^n)}\le C\bigl( \|g\|_{C^\alpha(\R^n \setminus B_1)}+C_0\bigr).\]
Therefore, since $\|\bar u^\epsilon \|_{C^\alpha(\R^n)}\le \|u^\epsilon \|_{C^\alpha(\R^n)}$, applying Theorem \ref{thm} to the function $\bar u^\epsilon \in C^{\sigma+\alpha}(B_1)$  we obtain the estimate
\begin{equation}\label{AAAAA}
 \|\bar u^\epsilon \|_{C^{\sigma+\alpha}(B_{1/2})}\le C\bigl(C_0 + \|g\|_{C^\alpha(\R^n\setminus B_1)} \bigr).
\end{equation}
Since $u^{\epsilon_m} \rightarrow u$ uniformly in $B_1$, it follows that $\bar u^{\epsilon_m} \rightarrow \bar u$ uniformly in $B_1$ and thus, passing \eqref{AAAAA} to the limit we find
\[
 \|\bar u \|_{C^{\sigma+\alpha}(B_{1/2})}\le C\bigl(C_0 + \|g\|_{C^\alpha(\R^n\setminus B_1)} \bigr).
\]
This implies that $u$ is $C^{\sigma+\alpha}$ in $B_{\rho/2}(z)$ --- since $u(x)= \bar u\bigl(\frac{x-z}{\rho}\bigr)$. Since all these balls $B_{\rho/2}(z)$ cover $B_1$ we have $u\in C^{\sigma+\alpha}(B_1)$.
Moreover,  when we take $z=0$ and $\rho =1$ we then  have $\bar u\equiv u$ and  we the previous estimate for $\tilde u$  yields the desired estimate for $\|u\|_{C^{\sigma+\alpha}(B_{1/2})}$.

In the case (b), using the trivial barriers we prove that
\[ \|\bar u^\epsilon \|_{L^\infty(\R^n)} = \|u^\epsilon \|_{L^\infty(\R^n)} \le C\bigl(\|g\|_{L^\infty(\R^n\setminus B_1)}+ C_0 \bigl).\]
Therefore, using Corollary \ref{Lalpha} applied to the function $\bar u$ we obtain that
\[ \|\bar u^\epsilon \|_{C^{\sigma+\alpha}(B_{1/2})}\le C\bigl(C_0 + \|g\|_{L^\infty(\R^n\setminus B_1)} \bigr)\]
where $C_0$ is the constant from \eqref{nontrinv2} and  $C$ depends only on $n$ $\sigma$,  $\alpha$, $\lambda$,  $\Lambda$, and $A_0$.
Again, this implies that $u\in C^{\sigma+\alpha}(B_1)$ and the estimate for $\|u\|_{C^{\sigma+\alpha}(B_{1/2})}$.

Finally, in both cases (a) and (b), after having proved the existence of a classical solution (attaining continuously the complement data),  its uniqueness among the class of viscosity solutions
follows from the trivial comparison principle between a classical solution and a viscosity solution.
\end{proof}

\section{Counterexamples to $C^{\sigma+\alpha}$ regularity for merely bounded complement data}
\label{SecCounter}

In this section we find sequences  $u_m$ of solutions to equations with rough kernels in $B_1$ that satisfy $\|u_m\|_{L^\infty(\R^n)}\le C$ and $\|u_m\|_{C^{\sigma+\alpha}(B_{1/2})} \nearrow \infty$ as $m\to \infty$ for all $\alpha>0$.
We consider the case of a linear equation and the case of a nonlinear convex equation involving the extremal operator $M^+_{\mathcal L_0}$.
Such sequences can be regarded as counterexamples to a $C^{\sigma+\alpha}$ interior estimate for linear or convex equations with rough kernels with merely bounded complement data.
These counterexamples  are built here in dimension $n=1$. Clearly, looking at these one-dimensional counterexamples as 1-D profiles in $\R^n$ we will have counterexamples in every dimension.

We will need the following elementary
\begin{claim}\label{claim1}
Assume that some function $u$ and $\alpha>0$ it is
\[ \| u\|_{C^{\sigma+\alpha}(-1/2,1/2)} \le C_0\]
and
\[  \| u\|_{C^\alpha(\R)} \le C_0.\]

Then, for all $L\in \mathcal L_0$ we have
\begin{equation} \label{abcde}\|Lu\|_{C^{\alpha'}(-1/4,1/4)}\le CC_0,
\end{equation}
where $\alpha'>0$ and $C$ depend only on $\sigma$ and ellipticity constants.
\end{claim}
\begin{proof}
We have
\[\bigl| \delta_2 u(x_1,y) - \delta_2 u(x_2,y)\bigr|
\le
\begin{cases}
C C_0 |y|^{\sigma+\alpha} \\
C_0 |x_1-x_2|^\alpha\\
C C_0|y|^{\theta(\sigma+\alpha)}|x_1-x_2|^{(1-\theta)\alpha}.
\end{cases}
\]
The third bound is obtained by ``interpolating''  the first and the second ones.

But then for all $x_1, x_2\in (-1/4,1/4)$
\[
\begin{split}
|L  u(x_1)-  L  u (x_2)| &\le C\int_{\R} \frac{\bigl| \delta_2  u (x_1,y) - \delta_2  u (x_2,y)\bigr|}{|y|^{n+\sigma}}\,dy\\
&\le  C \int_{-1}^1 \frac{|y|^{\theta(\sigma+\alpha)}|x_1-x_2|^{(1-\theta)\alpha}}{|y|^{1+\sigma}} \,dy + \int_{\R\setminus (-1,1)} \frac{|x_1-x_2|^{\alpha}}{|y|^{1+\sigma}} \,dy \\
&\le C |x_1-x_2|^{\alpha'},
\end{split}
\]
where we have taken $\theta<1$ very close to $1$ such that $\theta(\sigma+\alpha)>\sigma$ and $\alpha' = (1-\theta)\alpha$.
\end{proof}
We note that with the same assumptions of the Claim it is possible (and not difficult) to prove that \eqref{abcde} holds for $\alpha'=\alpha$ but the previous rough version will suffice for our purposes.

\subsection{Linear equations with rough kernels}

In $\R$,  for every  integer $m\ge1$ consider the function $u_m$ that solves
\[
\begin{cases}
L_m u_m= 0 \quad &\mbox{in }(-1,1)\\
u_m= 0 \quad & \mbox{in }[-2,-1]\cup [1,2]\\
{\rm sign} \sin(m\pi x) & \mbox{in }(-\infty,-2]\cup[2,\infty),
\end{cases}
\]
where $L_m$ is defined by
\[L_m v = \int_{\R}  \delta^2v(x,y) K_m(y)\,dy\]
for
\[
K_m(y)=
\begin{cases}
|y|^{-1-\sigma}\quad& \mbox{in }(-1,1)\\
\bigl(2+{\rm sign}\cos(m\pi y)\bigr) |y|^{-1-\sigma} & \mbox{in }(-\infty, -1)\cup (1,+\infty).
\end{cases}
\]

Next we use that for $p>0$ small enough the function $\psi(x)={\rm dist}\,\bigl(x, [-1/4,1/4]\bigr)^p$ is a supersolution in $(-1/4-\epsilon,-1/4)\cup (1/4,1/4+\epsilon)$,
for some $\epsilon>0$.  Namely, it satisfies $M^+_{\mathcal L_0} \psi \le 0$ there ---see \cite{CS2}.
Since $u_m\equiv 0$ in $(-2,-1)\cup(1,2)$, by using translates of $\psi$ (respectively $-\psi$) as upper  (lower) barrier we prove that
\[  |u_m| \le C {\rm dist}\bigl(x, (-\infty,-1]\cup[1,\infty)\bigr)^p \quad \mbox{in }(-1,1).\]
Combining this with known interior estimates we obtain, for small enough $\alpha>0$,
\[ \|u_m\|_{C^\alpha([-1,1])} \le C, \quad \mbox{for all } m.\]

Finally let us show that it is impossible that $\|u_m\|_{C^{\sigma+\alpha}(-1/2,1/2)}\le C$ with $\alpha>0$ and $C$ independent of $m$.
Let us write $u_m = u_m^{(1)} + u_m^{(2)}$, where $u_m^{(1)} = u_m \chi_{(-1,1)}$ and
\[ u_m^{(2)} =
\begin{cases}
0 \quad &\mbox{in }(-2,2)\\
{\rm sign} \sin(m\pi x) & \mbox{in }(-\infty,-2]\cup[2,\infty).
\end{cases}\]
We would then have
\[\| u_m^{(1)}\|_{C^{\sigma+\alpha}(-1/2,1/2)} + \|u_m^{(1)}\|_{C^\alpha(\R)} \le C.\]
Thus using Claim \ref{claim1} we would obtain
\[\| L_m u_m^{(1)}\|_{C^{\alpha'}(-1/4,1/4)} \le C,\]
 with $C$ independent of $m$.

Next, on the other hand
\[ L_m u_m^{(2)}(0) = 0 \quad \mbox{by odd symmetry of $u_m^{(2)}$}.\]
Let us now compute $L_m u_m^{(2)}\left( {\textstyle\frac{1}{2m}}\right) $. We we that for $|y|>2+\frac{1}{2m}$ we have
\[
\begin{split}
\delta^2 u_m^{(2)} \left( {\textstyle\frac{1}{2m}}, y\right)&= \frac{1}{2} {\textstyle \bigl\{ {\rm sign}\sin\bigl(\frac \pi 2 +m\pi y \bigr) + {\rm sign}\sin\bigl(\frac \pi 2 - m\pi y \bigr) - 2u_m^{(2)}\bigl(\frac 1 {2m}\bigr)\bigr\}}
\\
&=  {\rm sign}\cos\bigl( m\pi y \bigr)-0.
\end{split}
\]
We thus obtain
\[\begin{split}
L_m u_m^{(2)}\left( {\textstyle\frac{1}{2m}}\right) &= \int_{\R\setminus (-2,2)}   {\rm sign}\cos(m\pi y)  \,\frac{2+{\rm sign}\cos(m\pi y)}{|y|^{1+\sigma}}\,dy   + O(1/m)\\
&= c +  2\int_{\R\setminus (-2,2)}   \frac{{\rm sign}\cos(m\pi y)}{|y|^{1+\sigma}}\,dy + O(1/m)\\
&= c + o(1) \quad \mbox{as }m\nearrow \infty,
\end{split}\]
where $c = \int_{\R\setminus (-2,2)} |y|^{-1-\sigma}\,dy >0$.

Therefore,
\[\begin{split}
0  &= L_m u_m\left( {\textstyle\frac{1}{2m}}\right) - L_m u_m (0)
\\
&\ge L_m u_m^{(2)}\left( {\textstyle\frac{1}{2m}}\right) - L_m u_m^{(2)} (0)  -   \bigl|L_m u_m^{(1)}\left( {\textstyle\frac{1}{2m}}\right) - L_m u_m^{(1)} (0)\bigr|
\\
&\ge c - o(1) - C\left( {\textstyle\frac{1}{2m}}\right)^{\alpha'} \quad \mbox{as }m\nearrow \infty;
\end{split}\]
a contradiction.

\subsection{Nonlinear convex equation equation with the $M^+_{\mathcal L_0}$}
This is a variation of the previous example.
In $\R$,  for every $m\ge1$ consider the function $u_m$  that solves
\[
\begin{cases}
M^+_{\mathcal L_0} u_m= 0 \quad &\mbox{in }(-1,1)\\
u_m= 0 \quad & \mbox{in }[-2,-1]\cup [1,2]\\
{\rm sign} \sin(m\pi x) & \mbox{in }(-\infty,-2]\cup[2,\infty).
\end{cases}
\]

Since $-1 \le u_m\le 1$ in $\R$ but $\bigl|\{ u_m < 0\}\cap (-5,5)\bigr| \ge 1$ and $M^+_{\mathcal L_0} u=0$ in $(-1,1)$, it will be
\begin{equation}\label{latheta}
\-1\le  u_m \le 1-\tau \quad \mbox{in } [-1/2,1/2]
\end{equation}
for some $\tau>0$ depending only on $\sigma$ and ellipticity constants.

In addition, we use as in the previous subsection that for $p>0$ small enough the function $\psi(x)={\rm dist}\,\bigl(x, [-1/4,1/4]\bigr)^p$ is a supersolution in $(-1/4-\epsilon,-1/4)\cup (1/4,1/4+\epsilon)$,
for some $\epsilon>0$. Namely, it satisfies $M^+_{\mathcal L_0} \psi \le 0$ there.
Since $u_m\equiv 0$ in $(-2,-1)\cup(1,2)$, by using translates of $\psi$ (respectively $-\psi$) as upper  (lower) barrier we prove that
\[  |u_m| \le C {\rm dist}\bigl(x, (-\infty,-1]\cup[1,\infty)\bigr)^p\quad \mbox{in }(-1,1).\]
Combining this with known interior estimates we obtain, for small enough $\alpha>0$,
\begin{equation}\label{Calpha-all-m}
 \|u_m\|_{C^\alpha([-2,2])}  =  \|u_m\|_{C^\alpha([-1,1])} \le C, \quad \mbox{for all } m.
 \end{equation}

Next we use that for $|x|>2$  we have $u(x)={\rm sign} \sin(m\pi x)$, which is odd, we obtain
\begin{equation}\label{delta20}
 \delta^2u_m(0,y) = -u_m(0) \le 0 \quad  \mbox{for }|y|>2.
 \end{equation}
The fact that $u_m(0)\ge0$ can be easily deduced by observing that for all $L\in\mathcal L_0$ the  solution to the linear equation $Lw=0$ in $(-1,1)$ with the same boundary data as $u_m$ satisfies $w(0)=0$ (by odd symmetry), and $w$ it is a subsolution to our equation since
$M^+_{\mathcal L_0} w \ge Lw = 0$.

Let us denote
\[ b_m(y) =  \Lambda\bigl(\delta^2 u_m(0,y)\bigr)^+ +\lambda\bigl(\delta^2 u_m(0,y)\bigr)^-,  \]
that is,
\begin{equation}\label{M+0}
 M^+_{\mathcal L_0} u_m (0) = \int_\R \delta^2 u(0,y) \frac{b_m(y)}{|y|^{n+\sigma}}dy.
\end{equation}
Hence, by \eqref{delta20},
\begin{equation}\label{b}
 b_m (y)\equiv \lambda\quad \mbox{for }|y|>2.
 \end{equation}

Instead, at $x= \frac{1}{2m}$ we have
\begin{equation}\label{delta2P}
\delta^2u_m\left({\textstyle \frac{1}{2m}},y\right) =   {\rm sign}\cos(m\pi y)  -u_m\left({\textstyle \frac{1}{2m}}\right) \quad \mbox{ for }y\in\R\setminus \left(-2-{\textstyle \frac{1}{2m}}, 2-{\textstyle \frac{1}{2m}}\right).
\end{equation}
Hence if we let
\[ \tilde b_m(y) =  \Lambda\left(\delta^2 u_m\left({\textstyle \frac{1}{2m}},y\right)\right)^+ +\lambda\left(\delta^2 u_m\left({\textstyle \frac{1}{2m}},y\right)\right)^-,  \]
that is,
\begin{equation}\label{M+P}
M^+_{\mathcal L_0} u_m \left({\textstyle \frac{1}{2m}}\right) = \int_\R \delta^2  u_m\left({\textstyle \frac{1}{2m}},y\right) \frac{\tilde b_m(y)}{|y|^{n+\sigma}}dy.
\end{equation}
we then have
\begin{equation}\label{tildeb}
\tilde b_m (y) = \lambda + \frac{\Lambda-\lambda}{2}(1+{\rm sign}\,\cos (m\pi y))
\quad \mbox{ for }y\in \R\setminus \left(-2-{\textstyle \frac{1}{2m}}, 2-{\textstyle \frac{1}{2m}}\right).
\end{equation}

For all $\gamma\in(0,1)$ and for all $m$ using that $u_m(0) \in [0,1-\tau]$ by  \eqref{latheta}  we obtain
\[
 \begin{split}
\int_{\R\setminus(-2-\gamma,2+\gamma)} \bigl({\rm sign}\cos(m\pi y)  -u_m(0) \bigr) \frac{\frac{\Lambda-\lambda}{2}(1+{\rm sign}\,\cos (m\pi y)) }{ |y|^{n+\sigma}} \,dy  \ge  2 c_1 >0
\end{split}
\]
where $c_1$ is independent on $m$ ---like $\tau$.
Therefore, for all $m$ large enough so that $\left|\int_{\R\setminus(-2-\gamma,2+\gamma)} {\rm sign}\cos(m\pi y) \frac{\lambda}{ |y|^{n+\sigma}} \,dy\right| \le c_1 $ we have
\begin{equation}\label{needaname}
 \begin{split}
&\int_{\R\setminus(-2-\gamma,2+\gamma)} \bigl({\rm sign}\cos(m\pi y)  -u_m(0) \bigr) \frac{\lambda + \frac{\Lambda-\lambda}{2}(1+{\rm sign}\,\cos (m\pi y)) }{ |y|^{n+\sigma}} \,dy  \ge\\
&\hspace{50pt}\ge  c_1- \int_{\R\setminus(-2-\gamma,2-\gamma)}  u_m(0)  \frac{\lambda}{ |y|^{n+\sigma}}\,dy.
\end{split}
\end{equation}

Next, from \eqref{needaname}, using \eqref{delta20}, \eqref{b}, \eqref{delta2P}, and \eqref{tildeb} we obtain (for $m$ large enough, in particular $\frac{1}{2m}<\gamma$)
\begin{equation}\label{diffpositive}
 \begin{split}
& \int_{\R\setminus(-2-\gamma,2+\gamma)}  \delta^2u\left({\textstyle \frac{1}{2m}},y\right) \frac{\tilde b(y)}{|y|^{n+\sigma}}\,dy  -  \int_{\R\setminus(-2-\gamma,2+\gamma)}    \delta^2u(0,y) \frac{b(y)}{|y|^{n+\sigma}}\,dy\ge\\
&\hspace{20pt}\ge  c_1 -\left|  u_m\left({\textstyle \frac{1}{2m}}\right)- u_m(0) \right| \int_{\R\setminus(-2-\gamma,2+\gamma)} \frac{\lambda + \frac{\Lambda-\lambda}{2}(1+{\rm sign}\,\cos (m\pi y)) }{ |y|^{n+\sigma}} \,dy
\\
&\hspace{20pt}\ge  c_1 - C\left|{ \textstyle \frac{1}{2m}}- 0 \right|^\alpha.
\end{split}
\end{equation}
In the last inequality we have used \eqref{Calpha-all-m}.

Let us show that it is impossible that $\|u_m\|_{C^{\sigma+\alpha}(-1/2,1/2)}\le C$, with $\alpha>0$ and $C$ independent of $m$.

To reach a contradiction let us show that the kernels $b_m(y)|y|^{-n-\sigma}$ and $\tilde b_m(y)|y|^{-n-\sigma}$  would ``perform similarly'' when integrated  only in $(-2,2)$ against $\delta^2 u_m$ at the points $0$ and ${\textstyle \frac{1}{2m}}$.
More precisely, let us prove the bound
\begin{equation}\label{thebound}
\left| \int_{(-2+\gamma,2-\gamma)} \delta^2u(0,y) \frac{b(y)}{|y|^{n+\sigma}}\,dy -  \int_{(-2+\gamma,2-\gamma)} \delta^2u\left({\textstyle \frac{1}{2m}},y\right) \frac{\tilde b(y)}{|y|^{n+\sigma}}\,dy \right| \le C (1/m)^{{\alpha'}},
\end{equation}
for some ${\alpha'}\in (0,\alpha)$.
To prove \eqref{thebound} we use that that for $1/m<\gamma$ and $|y|<2-\gamma$ we would have
\[ \left|\delta^2u_m(0,y) -\delta^2u_m\left( {\textstyle \frac{1}{2m}},y\right) \right| \le
\begin{cases}
C |y|^{\sigma+\alpha} \\
C \left|0-{\textstyle \frac{1}{2m}}\right|^{\alpha}.
\end{cases}
\]
The first bound is obtained from the assumption $\|u_m\|_{C^{\sigma+\alpha}(-1/2,1/2)}\le C$ and the second from \eqref{Calpha-all-m}.
Hence, ``interpolating'' the two bounds we obtain
 \[
 \left|\delta^2u_m(0,y) -\delta^2u_m\left( {\textstyle \frac{1}{2m}},y\right) \right| \le C_1|y|^{\theta(\sigma+\alpha)} \left| {\textstyle \frac{1}{2m}}\right|^{(1-\theta)\alpha} = C_1|y|^{\sigma+{\alpha'}} \left|{\textstyle \frac{1}{2m}}\right|^{{\alpha'}},
 \]
where we have taken $\theta\in (0,1)$ close to $1$ and ${\alpha'}>0$ so that $\theta(\sigma+\alpha)= \sigma+{\alpha'}$ and $(1-\theta)\alpha={\alpha'}$.

Therefore if we split the interval $(-2+\gamma,2-\gamma)$ into the thee disjoint subsets
\[\boldsymbol A= \left\{ y\in (-2+\gamma,2-\gamma)\,:\,\delta^2u_m(0,y) \ge 2C_1 |y|^{\sigma+{{\alpha'}}} \left| {\textstyle \frac{1}{2m}}\right|^{{\alpha'}}  \right\}, \]
\[\boldsymbol B= \left\{ y\in (-2+\gamma,2-\gamma)\,:\,\delta^2u_m(0,y) \le- 2C_1 |y|^{\sigma+{{\alpha'}}} \left| {\textstyle \frac{1}{2m}}\right|^{{\alpha'}}  \right\},\]
 and
\[\boldsymbol C = (-2+\gamma,2-\gamma)\setminus(\boldsymbol A\cup \boldsymbol B).\]
Then it is $b =\tilde b = \Lambda$ in $\boldsymbol A$, $b=\tilde b = \lambda$ in $\boldsymbol B$ and
\[|\delta^2u_m(0,y)|+ \left|\delta^2u_m\left({\textstyle \frac{1}{2m}},y\right)\right| \le 4C_1 |y|^{\sigma+{{\alpha'}}} (1/m)^{{\alpha'}} \quad \mbox{ in } \boldsymbol C.\]
This clearly implies \eqref{thebound}.

Finally, using \eqref{M+0}, \eqref{M+P}, \eqref{diffpositive} , and \eqref{thebound}, we obtain
\[
\begin{split}
0 &= M^+_{\mathcal L_0} u_m \left({\textstyle \frac{1}{2m}}\right)  - M^+_{\mathcal L_0} u_m (0) \\
 &\ge c_1  -C (1/m)^{\alpha} -C(1/m)^{{\alpha'}} - C\gamma,
\end{split}
\]
which yields to contradiction taking first $\gamma < c_1/C$ and then $m$ large enough.

\section*{Acknowledgements}
The author is indebted to Xavier Cabr\'e, Hector Chang Lara, and Luis Silvestre for their interesting comments on the paper.

\end{document}